\documentclass[submission]{eptcs}

\usepackage{iftex}

\usepackage{pgfgantt} 

\usepackage{xcolor}
\definecolor{parchment}{RGB}{214, 204, 169}

\usepackage{microtype}
\usepackage{verbatim}
\usepackage{graphicx}
\usepackage{rotating}

\usepackage{svg}

\usepackage{mathtools}

\usepackage{url}
\definecolor{linkColor}{RGB}{156,78,13}
\usepackage{hyperref}
\hypersetup{
    colorlinks = true,
    allcolors = linkColor 
    }
\usepackage{pslatex}

\usepackage{datetime}
\newdateformat{versiondate}{%
\THEMONTH\THEDAY}

\newcommand{\coprodcup}{\mathbin{\rotatebox[origin=c]{180}{$\Pi$}}}

\usepackage{pgf,tikz,environ}
\usetikzlibrary{arrows}
\usetikzlibrary{cd}

\usepackage{xpatch} 
\usepackage{sectsty}
\newcommand{\secfont}{\fontfamily{lmss}\selectfont}
\allsectionsfont{\secfont}
\usepackage{amsthm}
\usepackage{amsmath} 
\usepackage{faktor}

\usepackage{adjustbox}

\newcommand{\defeq}{\vcentcolon=}

\newtheoremstyle{zoltanstyle}
  {1em} 
  {\topsep} 
  {} 
  {} 
  {\bfseries} 
  {.} 
  {.5em} 
  {} 

\theoremstyle{zoltanstyle}
\swapnumbers
\xpatchcmd\swappedhead{~}{.~}{}{}
\newtheorem{body}{}
\numberwithin{body}{section}

\newtheorem{corollary}[body]{\secfont Corollary}
\newtheorem{definition}[body]{\secfont Definition}
\newtheorem{example}[body]{\secfont Example}

\newtheorem{lemma}[body]{\secfont Lemma}

\newtheorem{proposition}[body]{\secfont Proposition}
\newtheorem{remark}[body]{\secfont Remark}
\newtheorem{observation}[body]{\secfont Observation}
\newtheorem{theorem}[body]{\secfont Theorem}

\usepackage{pgf,tikz,environ}
\usepackage{quiver}
\usetikzlibrary{arrows}
\usetikzlibrary{cd}
\usetikzlibrary{positioning}
\usetikzlibrary{fadings}
\usetikzlibrary{decorations.pathreplacing}
\usetikzlibrary{decorations.pathmorphing}
\tikzset{snake it/.style={decorate, decoration=snake}}

\definecolor{parchment}{RGB}{214, 204, 169}

\expandafter\let\expandafter\oldproof\csname\string\proof\endcsname
\let\oldendproof\endproof

\renewenvironment{proof}[1][\proofname]{%
  \oldproof[\normalfont \bfseries #1.]%
}{\oldendproof}

\usepackage{enumitem}
\usepackage{mathptmx}

\usepackage[T1]{fontenc}
\usepackage[utf8]{inputenc}

\usetikzlibrary {positioning}
\usetikzlibrary{fadings}
\usetikzlibrary{decorations.pathreplacing}
\usetikzlibrary{decorations.pathmorphing}
\tikzset{snake it/.style={decorate, decoration=snake}}
\definecolor{gold}{RGB}{255,215,0}
\definecolor{softBlack}{RGB}{45, 47, 49}
\definecolor{creamWhite}{RGB}{245,244,241}
\definecolor{softGray}{RGB}{220, 216, 214}
\definecolor{brick}{RGB}{232, 48, 48}

\newcommand{\typesetoperator}[1]{\mathsf{#1}}

\usepackage{xspace}
\definecolor{gold}{RGB}{255,215,0}
\definecolor{softBlack}{RGB}{45, 47, 49}
\definecolor{creamWhite}{RGB}{245,244,241}
\definecolor{softGray}{RGB}{220, 216, 214}
\definecolor{brick}{RGB}{232, 48, 48}

\usepackage{enumitem}

\newcommand{\define}[1]{\textbf{#1}}

\newcommand{\cat}[1]{\mathsf{#1}}

\DeclareMathOperator{\op}{^{\typesetoperator{op}}}

\DeclareMathOperator{\sub}{\typesetoperator{Sub}}

\DeclareMathOperator{\im}{\typesetoperator{im}}
\DeclareMathOperator{\imageObj}{\typesetoperator{im}}
\DeclareMathOperator{\id}{\typesetoperator{id}}

\DeclareMathOperator{\coker}{\typesetoperator{coker}}
\let\ker\relax
\DeclareMathOperator{\ker}{\typesetoperator{ker}}

\DeclareMathOperator*{\colim}{\typesetoperator{colim}}

\newcommand{\typesetproblem}[1]{\textsc{#1}}

\DeclareMathOperator{\fpt}{\typesetproblem{FPT}}

\DeclareMathOperator{\Ua}{\mathcal{U}}  
\DeclareMathOperator{\Va}{\mathcal{V}}

\newcommand{\Z}{\mathbb{Z}}

\newcommand{\restrict}{\!\!\upharpoonright}

\newcommand{\Set}{\cat{Set}}
\newcommand{\Sh}{\cat{Sh}}
\newcommand{\Psh}{\cat{Psh}}
\newcommand{\SPsh}{\cat{SPsh}}
\newcommand{\Cov}{\cat{Cov}}
\newcommand{\Match}{\cat{Match}}

\usepackage{newtxtext}
\usepackage{newtxmath}

\newcommand{\defproblem}[3]{
	\vspace{1mm}
	\noindent\fbox{
		\begin{minipage}{0.96\textwidth}
			\begin{tabular*}{\textwidth}{@{\extracolsep{\fill}}lr} \textsc{#1} &  \\ \end{tabular*}
			{\bf{Input:}} #2 \\
			{\bf{Question:}} #3
		\end{minipage}
	}
	\vspace{1mm}
}

\ifpdf
  \usepackage{underscore}         
  \usepackage[T1]{fontenc}        
\else
  \usepackage{breakurl}           
\fi

\title{Algorithmic and Extremal Obstructions Through the Language of Cohomology}

\def\titlerunning{Failures of Compositionality}
\def\authorrunning{Azevedo, Bumpus, Capucci, Fairbanks \& Rosiak}
\author{Anny Beatriz Azevedo
    \institute{University of São Paulo}
    \and
    Benjamin Merlin Bumpus 
    \institute{University of São Paulo}
    \and 
    Matteo Capucci
    \institute{University of Strathclyde}
    \and
    James Fairbanks
    \institute{University of Florida}
    \and
    Daniel Rosiak
    \institute{NIST}
}

\begin{document}

\maketitle

\begin{abstract}
    We model problems as presheaves that assign sets of certificates to input instances and we show how to use presheaf Čech cohomology to capture the precise ways in which local solutions fail to patch into global ones. Applied to problems like \textsc{VertexCover}, \textsc{CycleCover}, and \textsc{OddCycleTransversal}, our framework exposes emergent phenomena, such as hidden cycles or the inflation of small, local solutions. This approach not only rephrases classical results like König’s Theorem in cohomological terms, but also reveals how to systematically account for failures of compositionality. Although our main focus is on presheaves of sets, the methods generalize naturally to Abelian presheaves, suggesting a rich interplay between graph theory, cohomology, and complexity. This work represents a first step toward a systematic, sheaf-theoretic theory of algorithmic structure and related obstructions.
\end{abstract}

\section{Introduction}\label{sec: intro}

\body{
    Obstructions, objects whose existence precludes the possibility of a given mathematical construction, abound in mathematics and, accordingly, many deep results throughout the field are concerned with finding, classifying and accounting for them. Indeed, there is no lack of preeminent conjectures -- such as the \textit{Hodge Conjecture} (a Millennium Prize problem), the \textit{Tate Conjecture} and the \textit{Weil Conjectures}, to name a few famous ones -- that deal more or less directly with studying obstructions.
}

\body{
    In modern, or parameterized complexity theory, the importance of understanding combinatorial obstructions is elucidated by the simple, but profound observation that $\fpt$ \footnote{The class of Fixed-Parameter Tractable problems which are solvable by algorithms that are exponential only in the size of a fixed parameter while polynomial in the size of the input.} is equivalent to polynomial-time pre-processing, or \textit{kernelization}~\cite{downey2012parameterized}. Hence there is a sense in which efficient algorithms are deeply related to either efficiently \textit{computing obstructions} or efficiently reducing an input instance to a small graph. Thus not only are obstructions mathematically relevant, practical and computationally useful, but, as Estivill-Castro, Fellows, Langston and Rosamond point out~\cite{fellows-ptime-extremal-structure}, they stand out as a promising ingredient towards the development of a `mathematical monster machine'~\cite{summoned-from-the-void} -- where theoretical advances are obtained through small, principled, nearly automated steps -- for the development of algorithmics. 
}

\body{
    Thus it would be useful to have an abstract perspective on obstructions in combinatorics and algorithmics. But how, then, in these settings should we go about detecting, accounting and manipulating obstructions in an organized and systematic way? Historically, both in the context of the aforementioned famous conjectures and in many other application areas, \textit{cohomology} stands out as a convenient tool to get the job done~\cite{GallierQuaintanceBook}. Other than areas of discrete mathematics, that are entirely built around cohomological methods (e.g.~\textit{topological data analysis} and \textit{discrete Morse Theory}), classical combinatorics also touts examples of important results which rely indispensably on elegant (co)-homological methods~\cite{topologicalcombinatorics,  matouvsek2003using, scoville2019discrete}. Two specific and celebrated examples include Lovász's resolution of Kneeser's conjecture~\cite{lovasz1978kneser} and Alon's result on splitting necklaces~\cite{alon1987splitting}.
}

\body{
    In this paper, we initiate a systematic account of obstructions arising in algorithmics and we do so by cohomological means. In particular, we find these tools well suited for speaking about both \textit{obstructions to the existence of a solution} and \textit{obstructions to compositionality}. We give a construction that allows us to cohomologically rephrase (but not prove) König's Theorem and other results in the flavor of the Erdös-Pósa. Furthermore, we show that one can use \textit{presheaf} Čech cohomology to encode precisely the kinds of failures that occur when one tries to patch together local solutions into global ones. 
}

\body{
    Our perspective is to view computational problems as structured data assignments attaching sets of certificates to input instances. Said in light category-theoretic language, we view computational problems as \textit{presheaves} and we study their \textit{presheaf} Čech cohomology. Applied to \textsc{VertexCover}, \textsc{CycleCover} and \textsc{OddCycleTransversal}, for instance, the zeroth \textit{presheaf} Čech cohomology group accounts for all \textit{emergent obstructions} such as local small solutions becoming too large when joined together or the emergence of (odd) cycles that are not visible in the small constituent subgraphs comprising a (odd) cycle cover. Furthermore, we develop some cohomological results of independent interest that allow for the application of these methods to other computational problems without needing a deep understanding of cohomology.  
}

\body{
    Viewing computational problem as presheaves is not new to this contribution; indeed, in a short note~\cite{mazzatowards}, Mazza proposes a sheaf-theoretic outlook on complexity theory. Moreover, this perspective is not even foreign to discrete mathematics; for example in Dinur's \href{https://www.youtube.com/watch?v=eIEWMuxt2c8}{talk}\footnote{The talk was held at the Institute for Advanced Study in Princeton and sheaves are mentioned roughly at minute \href{https://www.youtube.com/watch?v=eIEWMuxt2c8}{4:10}.}~on her celebrated proof of the PCP Theorem~\cite{Dinur05}, she points out that the sheaf-theoretic structure of proper colorings of graphs plays a central role. More recently, there has been a growing use of sheaf-theoretic methods in the study of (promise) constraint satisfaction problems~\cite{deciding-sheaves, oconghaile, topological-Hell-Nesetril, filakovsky, Topology-in-PCSP} and distributed computation~\cite{malcolm2009sheaves, felber2025sheaf}.  
}

\body{Identifying the role of compositionality in certain problems is also not new: in applied category theory, there has been a thorough study of compositional structures and their applications~\cite{structured-cospans, decorated-cospans, baez-open-petri-nets, structured-decompositions, chaudhuri-mathematical-framework-biochemical-porcesses}. And along with that also comes a growing need for accounting for non-compositionality and the study of such obstructions~\cite{puca2023obstructions}.  
}

\body{ Although, as we have mentioned above, there has been a growing community which studies graph properties encoded as (pre)-sheaves, this nascent field is still young and there has not been a formal, systematic investigation of the expressiveness of cohomological methods in this setting. For this reason, here we focus on the simplest possible encoding of computational problems as presheaves; namely we consider \textit{presheaves of sets} which assign a \textit{sets of certificates} to each input graph. This choice has a few drawbacks and a few strengths. 
    \begin{itemize}
        \item \textit{Drawbacks:} since we are considering presheaves of sets (rather than, say, presheaves of vector spaces or modules), all of our cohomological computations rely on free Abelianizations. This has the effect that the higher cohomology groups are not of clear significance and thus the full power of homological algebra does not come to bear. 
        \item \textit{Strengths:} even in this comparatively simple setting, we obtain compelling insights into both compositional and structural obstructions to computational problems by studying the zeroth cohomology group. By standard abstract nonsense, these results are all easily seen to generalize immediately to presheaves valued in any Abelian category (i.e. presheaves of Abelian groups or vector spaces over a field or modules over a ring). 
    \end{itemize}
}

\body{
We believe that considering Abelian presheaves -- mappings that attach vector spaces, say, to input instances -- is an exciting and fruitful direction of future study and we suspect that the methods presented here will open new doors for the interplay of graph theory, algebra and topology. Accordingly, we endeavor to make the present article as accessible and self-contained as possible so as to accommodate readers with varied backgrounds.
}

\section{Motivation: Naive Dynamic Programming}

\body{\textit{Divide and conquer}, or \textit{dynamic programming}, is a common approach in algorithmics. The idea is as follows: to find solutions to computational problem on an input $G$, one works locally on small constituent parts thereof and then finds clever ways of patching these local solutions together into global solutions on $G$. A key step in designing these kinds of algorithms is that of identifying how local solutions may fail to patch together and how to go about keeping track of this information. This is what we mean by \textit{obstructions to compositionality} and, as mentioned in Section~\ref{sec: intro}, this article is focused on finding general mathematical language for speaking about such obstructions and mathematical tools for identifying them. It is our contention that \textit{cohomology} is a useful tool for accounting for exactly this kind of information. To that end, in this section we proceed by means of examples to set the scene for the rest of the paper. In particular, we will consider two computational problems (\textsc{Coloring} and \textsc{VertexCover}) and explain at an intuitive level the kinds of obstructions to dynamic programming that we will later on identify algebraically.
}

\body{Consider the following graph (left) broken up into three pieces (right) as follows.

\[
\tikzset{every picture/.style={line width=0.75pt}} 

\]

The goal is to decide if there exists a 3-coloring of the above graph by solving the problem on the pieces and combining this information into a solution on the whole. To have any hope of negotiating agreements between local data, we first need to be able to relate local solutions to each other. We do so by taking into account the intersection between the pieces, which are represented below, together with their obvious inclusions (which we denote by hooked arrows).

\[
\tikzset{every picture/.style={line width=0.75pt}} 
%
\]
}
\body{
Now the key intuition is that, if we know that each of the pieces above is small (bounded by some constant $k$, say), then we can enumerate the entire solution space locally and only incur a bounded timecost of $f(k)$ for each piece (where $f$ may well be exponential). For instance, in the case of 3-coloring, one has $f(k) = 3^k$ and this can be visualized on our example graph as follows (where for the sake of space we only draw some of the many solutions for the leftmost piece).  

\[
\tikzset{every picture/.style={line width=0.75pt}} 
%
\]

Solutions in each piece relate to each other through their restriction to the pairwise intersections of the pieces. This is represented above by the arrows pointing down. Notice that, in the case of $3$-\textsc{Coloring}, every global solution (below, right) gives rise to a set of local solutions (below, left) that have the same pairwise restrictions. But there is more: every such set of pairwise agreeing solutions gives rise to a global solution.

\[
\tikzset{every picture/.style={line width=0.75pt}} 
%
\]

In particular, this means that, if we find a local solution that never contributes to a family of the kind shown above (i.e of pairwise agreeing local solutions), then we can safely remove it for it will never contribute to a global solution.  
With this in mind, we can find solutions on the whole recursively as follows. We start by asking if there exist pairs of solutions $(c_1, c_2)$ of the first and second pieces respectively such that $c_1$ and $c_2$ have the same restriction: if no such pair exists, then we reject; otherwise, we remove from the sets of local solutions all those elements that cannot be paired up this way. We then proceed recursively with the pieces of the decomposed instance. By what we already observed, this suffices for $3$-\textsc{Coloring} since, if by the end of this process no local set of solutions has become empty, then we can can conclude that the whole admits a proper $3$-coloring.
} \label{algorithm-coloring}

\body{This is slightly non-standard way of writing down the well-known algorithm for coloring by dynamic programming on tree decompositions \cite{downey2013fundamentals, flum2006parameterized}. As it turns out, using the language of category theory one can generalize this algorithm for categories equipped with a notion of \textit{cover} of objects (in this case subgraphs whose union results in the bigger graph) as long as we can attach local solutions into a global one \footnote{For the reader knowledgeable in sheaf theory, one can generalize this algorithm for sufficiently well-behaved categories equipped with a subcanonical topology, as long as the data assignment is compositional, i.e., a sheaf.}, as is the case for coloring~\cite{deciding-sheaves}} (concretely, this amounts to a general categorical algorithm for constraint satisfaction problems). 

\body{In contrast, in the case of \textsc{VertexCover$_2$}, we will see that local solutions are not as well-behaved. We employ the same instance along as the same decomposition into local pieces as before, but now we wish to decide if there exists a vertex cover of the graph of size at most $2$. Once again, assuming each part has at most $k$ vertices in it, we can enumerate the $2^k$ vertex covers of each piece as shown below (where the vertices chosen to be part of local vertex covers are marked in red).

\[
\tikzset{every picture/.style={line width=0.75pt}} 
\begin{tikzpicture}[x=0.65pt,y=0.65pt,yscale=-1,xscale=1]

\draw    (165.51,48) -- (194.36,76) ;
\draw    (165.51,48) -- (196.33,48) ;
\draw [color={rgb, 255:red, 0; green, 0; blue, 0 }  ,draw opacity=1 ]   (151.36,73) ;
\draw [shift={(151.36,73)}, rotate = 0] [color={rgb, 255:red, 0; green, 0; blue, 0 }  ,draw opacity=1 ][fill={rgb, 255:red, 0; green, 0; blue, 0 }  ,fill opacity=1 ][line width=0.75]      (0, 0) circle [x radius= 3.35, y radius= 3.35]   ;
\draw    (283.62,63.03) -- (282.75,102.09) ;
\draw [color={rgb, 255:red, 208; green, 2; blue, 27 }  ,draw opacity=1 ]   (282.75,102.09) ;
\draw [shift={(282.75,102.09)}, rotate = 0] [color={rgb, 255:red, 208; green, 2; blue, 27 }  ,draw opacity=1 ][fill={rgb, 255:red, 208; green, 2; blue, 27 }  ,fill opacity=1 ][line width=0.75]      (0, 0) circle [x radius= 3.35, y radius= 3.35]   ;
\draw [color={rgb, 255:red, 0; green, 0; blue, 0 }  ,draw opacity=1 ]   (283.62,63.03) ;
\draw [shift={(283.62,63.03)}, rotate = 0] [color={rgb, 255:red, 0; green, 0; blue, 0 }  ,draw opacity=1 ][fill={rgb, 255:red, 0; green, 0; blue, 0 }  ,fill opacity=1 ][line width=0.75]      (0, 0) circle [x radius= 3.35, y radius= 3.35]   ;
\draw    (306.62,62.86) -- (305.75,101.92) ;
\draw [color={rgb, 255:red, 208; green, 2; blue, 27 }  ,draw opacity=1 ]   (306.62,62.86) ;
\draw [shift={(306.62,62.86)}, rotate = 0] [color={rgb, 255:red, 208; green, 2; blue, 27 }  ,draw opacity=1 ][fill={rgb, 255:red, 208; green, 2; blue, 27 }  ,fill opacity=1 ][line width=0.75]      (0, 0) circle [x radius= 3.35, y radius= 3.35]   ;
\draw [color={rgb, 255:red, 0; green, 0; blue, 0 }  ,draw opacity=1 ]   (305.75,101.92) ;
\draw [shift={(305.75,101.92)}, rotate = 0] [color={rgb, 255:red, 0; green, 0; blue, 0 }  ,draw opacity=1 ][fill={rgb, 255:red, 0; green, 0; blue, 0 }  ,fill opacity=1 ][line width=0.75]      (0, 0) circle [x radius= 3.35, y radius= 3.35]   ;
\draw    (328.62,62.7) -- (327.75,101.76) ;
\draw [color={rgb, 255:red, 208; green, 2; blue, 27 }  ,draw opacity=1 ]   (327.75,101.76) ;
\draw [shift={(327.75,101.76)}, rotate = 0] [color={rgb, 255:red, 208; green, 2; blue, 27 }  ,draw opacity=1 ][fill={rgb, 255:red, 208; green, 2; blue, 27 }  ,fill opacity=1 ][line width=0.75]      (0, 0) circle [x radius= 3.35, y radius= 3.35]   ;
\draw [color={rgb, 255:red, 208; green, 2; blue, 27 }  ,draw opacity=1 ]   (328.62,62.7) ;
\draw [shift={(328.62,62.7)}, rotate = 0] [color={rgb, 255:red, 208; green, 2; blue, 27 }  ,draw opacity=1 ][fill={rgb, 255:red, 208; green, 2; blue, 27 }  ,fill opacity=1 ][line width=0.75]      (0, 0) circle [x radius= 3.35, y radius= 3.35]   ;
\draw [color={rgb, 255:red, 208; green, 2; blue, 27 }  ,draw opacity=1 ]   (369,239) ;
\draw [shift={(369,239)}, rotate = 0] [color={rgb, 255:red, 208; green, 2; blue, 27 }  ,draw opacity=1 ][fill={rgb, 255:red, 208; green, 2; blue, 27 }  ,fill opacity=1 ][line width=0.75]      (0, 0) circle [x radius= 3.35, y radius= 3.35]   ;
\draw   (323.82,237.68) .. controls (323.82,203.82) and (351.25,176.37) .. (385.08,176.37) .. controls (418.91,176.37) and (446.33,203.82) .. (446.33,237.68) .. controls (446.33,271.55) and (418.91,299) .. (385.08,299) .. controls (351.25,299) and (323.82,271.55) .. (323.82,237.68) -- cycle ;
\draw   (92.82,82.68) .. controls (92.82,48.82) and (120.25,21.37) .. (154.08,21.37) .. controls (187.91,21.37) and (215.33,48.82) .. (215.33,82.68) .. controls (215.33,116.55) and (187.91,144) .. (154.08,144) .. controls (120.25,144) and (92.82,116.55) .. (92.82,82.68) -- cycle ;
\draw   (302.63,21.43) .. controls (336.49,21.18) and (364.14,48.41) .. (364.39,82.23) .. controls (364.64,116.06) and (337.39,143.69) .. (303.53,143.94) .. controls (269.66,144.19) and (242.01,116.96) .. (241.76,83.13) .. controls (241.52,49.3) and (268.77,21.68) .. (302.63,21.43) -- cycle ;
\draw [color={rgb, 255:red, 0; green, 0; blue, 0 }  ,draw opacity=1 ]   (121.51,40) -- (151.36,73) ;
\draw    (121.51,40) -- (152.33,40) ;
\draw    (99.51,75) -- (129.36,108) ;
\draw [color={rgb, 255:red, 0; green, 0; blue, 0 }  ,draw opacity=1 ]   (99.51,75) -- (130.33,75) ;
\draw [shift={(130.33,75)}, rotate = 0] [color={rgb, 255:red, 0; green, 0; blue, 0 }  ,draw opacity=1 ][fill={rgb, 255:red, 0; green, 0; blue, 0 }  ,fill opacity=1 ][line width=0.75]      (0, 0) circle [x radius= 3.35, y radius= 3.35]   ;
\draw    (166.51,87) -- (196.36,120) ;
\draw    (166.51,87) -- (197.33,87) ;
\draw [color={rgb, 255:red, 208; green, 2; blue, 27 }  ,draw opacity=1 ]   (165.51,48) ;
\draw [shift={(165.51,48)}, rotate = 0] [color={rgb, 255:red, 208; green, 2; blue, 27 }  ,draw opacity=1 ][fill={rgb, 255:red, 208; green, 2; blue, 27 }  ,fill opacity=1 ][line width=0.75]      (0, 0) circle [x radius= 3.35, y radius= 3.35]   ;
\draw [color={rgb, 255:red, 208; green, 2; blue, 27 }  ,draw opacity=1 ]   (196.33,48) ;
\draw [shift={(196.33,48)}, rotate = 0] [color={rgb, 255:red, 208; green, 2; blue, 27 }  ,draw opacity=1 ][fill={rgb, 255:red, 208; green, 2; blue, 27 }  ,fill opacity=1 ][line width=0.75]      (0, 0) circle [x radius= 3.35, y radius= 3.35]   ;
\draw [color={rgb, 255:red, 208; green, 2; blue, 27 }  ,draw opacity=1 ]   (194.36,76) ;
\draw [shift={(194.36,76)}, rotate = 0] [color={rgb, 255:red, 208; green, 2; blue, 27 }  ,draw opacity=1 ][fill={rgb, 255:red, 208; green, 2; blue, 27 }  ,fill opacity=1 ][line width=0.75]      (0, 0) circle [x radius= 3.35, y radius= 3.35]   ;
\draw [color={rgb, 255:red, 0; green, 0; blue, 0 }  ,draw opacity=1 ]   (129.36,108) ;
\draw [shift={(129.36,108)}, rotate = 0] [color={rgb, 255:red, 0; green, 0; blue, 0 }  ,draw opacity=1 ][fill={rgb, 255:red, 0; green, 0; blue, 0 }  ,fill opacity=1 ][line width=0.75]      (0, 0) circle [x radius= 3.35, y radius= 3.35]   ;
\draw [color={rgb, 255:red, 208; green, 2; blue, 27 }  ,draw opacity=1 ]   (99.51,75) ;
\draw [shift={(99.51,75)}, rotate = 0] [color={rgb, 255:red, 208; green, 2; blue, 27 }  ,draw opacity=1 ][fill={rgb, 255:red, 208; green, 2; blue, 27 }  ,fill opacity=1 ][line width=0.75]      (0, 0) circle [x radius= 3.35, y radius= 3.35]   ;
\draw [color={rgb, 255:red, 208; green, 2; blue, 27 }  ,draw opacity=1 ]   (121.51,40) ;
\draw [shift={(121.51,40)}, rotate = 0] [color={rgb, 255:red, 208; green, 2; blue, 27 }  ,draw opacity=1 ][fill={rgb, 255:red, 208; green, 2; blue, 27 }  ,fill opacity=1 ][line width=0.75]      (0, 0) circle [x radius= 3.35, y radius= 3.35]   ;
\draw [color={rgb, 255:red, 208; green, 2; blue, 27 }  ,draw opacity=1 ]   (152.33,40) ;
\draw [shift={(152.33,40)}, rotate = 0] [color={rgb, 255:red, 208; green, 2; blue, 27 }  ,draw opacity=1 ][fill={rgb, 255:red, 208; green, 2; blue, 27 }  ,fill opacity=1 ][line width=0.75]      (0, 0) circle [x radius= 3.35, y radius= 3.35]   ;
\draw [color={rgb, 255:red, 208; green, 2; blue, 27 }  ,draw opacity=1 ]   (197.33,87) ;
\draw [shift={(197.33,87)}, rotate = 0] [color={rgb, 255:red, 208; green, 2; blue, 27 }  ,draw opacity=1 ][fill={rgb, 255:red, 208; green, 2; blue, 27 }  ,fill opacity=1 ][line width=0.75]      (0, 0) circle [x radius= 3.35, y radius= 3.35]   ;
\draw [color={rgb, 255:red, 208; green, 2; blue, 27 }  ,draw opacity=1 ]   (196.36,120) ;
\draw [shift={(196.36,120)}, rotate = 0] [color={rgb, 255:red, 208; green, 2; blue, 27 }  ,draw opacity=1 ][fill={rgb, 255:red, 208; green, 2; blue, 27 }  ,fill opacity=1 ][line width=0.75]      (0, 0) circle [x radius= 3.35, y radius= 3.35]   ;
\draw [color={rgb, 255:red, 0; green, 0; blue, 0 }  ,draw opacity=1 ]   (166.51,87) ;
\draw [shift={(166.51,87)}, rotate = 0] [color={rgb, 255:red, 0; green, 0; blue, 0 }  ,draw opacity=1 ][fill={rgb, 255:red, 0; green, 0; blue, 0 }  ,fill opacity=1 ][line width=0.75]      (0, 0) circle [x radius= 3.35, y radius= 3.35]   ;
\draw    (176.33,148) -- (193.22,173.34) ;
\draw [shift={(194.33,175)}, rotate = 236.31] [color={rgb, 255:red, 0; green, 0; blue, 0 }  ][line width=0.75]    (10.93,-3.29) .. controls (6.95,-1.4) and (3.31,-0.3) .. (0,0) .. controls (3.31,0.3) and (6.95,1.4) .. (10.93,3.29)   ;
\draw    (336.33,144) -- (353.22,169.34) ;
\draw [shift={(354.33,171)}, rotate = 236.31] [color={rgb, 255:red, 0; green, 0; blue, 0 }  ][line width=0.75]    (10.93,-3.29) .. controls (6.95,-1.4) and (3.31,-0.3) .. (0,0) .. controls (3.31,0.3) and (6.95,1.4) .. (10.93,3.29)   ;
\draw    (270,145) -- (255.19,176.19) ;
\draw [shift={(254.33,178)}, rotate = 295.4] [color={rgb, 255:red, 0; green, 0; blue, 0 }  ][line width=0.75]    (10.93,-3.29) .. controls (6.95,-1.4) and (3.31,-0.3) .. (0,0) .. controls (3.31,0.3) and (6.95,1.4) .. (10.93,3.29)   ;
\draw    (428.33,147) -- (418.02,175.12) ;
\draw [shift={(417.33,177)}, rotate = 290.14] [color={rgb, 255:red, 0; green, 0; blue, 0 }  ][line width=0.75]    (10.93,-3.29) .. controls (6.95,-1.4) and (3.31,-0.3) .. (0,0) .. controls (3.31,0.3) and (6.95,1.4) .. (10.93,3.29)   ;
\draw    (137.51,103) -- (167.36,136) ;
\draw    (137.51,103) -- (168.33,103) ;
\draw [color={rgb, 255:red, 0; green, 0; blue, 0 }  ,draw opacity=1 ]   (168.33,103) ;
\draw [shift={(168.33,103)}, rotate = 0] [color={rgb, 255:red, 0; green, 0; blue, 0 }  ,draw opacity=1 ][fill={rgb, 255:red, 0; green, 0; blue, 0 }  ,fill opacity=1 ][line width=0.75]      (0, 0) circle [x radius= 3.35, y radius= 3.35]   ;
\draw [color={rgb, 255:red, 208; green, 2; blue, 27 }  ,draw opacity=1 ]   (167.36,136) ;
\draw [shift={(167.36,136)}, rotate = 0] [color={rgb, 255:red, 208; green, 2; blue, 27 }  ,draw opacity=1 ][fill={rgb, 255:red, 208; green, 2; blue, 27 }  ,fill opacity=1 ][line width=0.75]      (0, 0) circle [x radius= 3.35, y radius= 3.35]   ;
\draw [color={rgb, 255:red, 208; green, 2; blue, 27 }  ,draw opacity=1 ]   (137.51,103) ;
\draw [shift={(137.51,103)}, rotate = 0] [color={rgb, 255:red, 208; green, 2; blue, 27 }  ,draw opacity=1 ][fill={rgb, 255:red, 208; green, 2; blue, 27 }  ,fill opacity=1 ][line width=0.75]      (0, 0) circle [x radius= 3.35, y radius= 3.35]   ;
\draw    (469.38,62.88) -- (430.32,62.7) ;
\draw [color={rgb, 255:red, 208; green, 2; blue, 27 }  ,draw opacity=1 ]   (430.32,62.7) ;
\draw [shift={(430.32,62.7)}, rotate = 0] [color={rgb, 255:red, 208; green, 2; blue, 27 }  ,draw opacity=1 ][fill={rgb, 255:red, 208; green, 2; blue, 27 }  ,fill opacity=1 ][line width=0.75]      (0, 0) circle [x radius= 3.35, y radius= 3.35]   ;
\draw [color={rgb, 255:red, 0; green, 0; blue, 0 }  ,draw opacity=1 ]   (469.38,62.88) ;
\draw [shift={(469.38,62.88)}, rotate = 0] [color={rgb, 255:red, 0; green, 0; blue, 0 }  ,draw opacity=1 ][fill={rgb, 255:red, 0; green, 0; blue, 0 }  ,fill opacity=1 ][line width=0.75]      (0, 0) circle [x radius= 3.35, y radius= 3.35]   ;
\draw    (469.96,85.87) -- (430.89,85.69) ;
\draw [color={rgb, 255:red, 208; green, 2; blue, 27 }  ,draw opacity=1 ]   (469.96,85.87) ;
\draw [shift={(469.96,85.87)}, rotate = 0] [color={rgb, 255:red, 208; green, 2; blue, 27 }  ,draw opacity=1 ][fill={rgb, 255:red, 208; green, 2; blue, 27 }  ,fill opacity=1 ][line width=0.75]      (0, 0) circle [x radius= 3.35, y radius= 3.35]   ;
\draw [color={rgb, 255:red, 0; green, 0; blue, 0 }  ,draw opacity=1 ]   (430.89,85.69) ;
\draw [shift={(430.89,85.69)}, rotate = 0] [color={rgb, 255:red, 0; green, 0; blue, 0 }  ,draw opacity=1 ][fill={rgb, 255:red, 0; green, 0; blue, 0 }  ,fill opacity=1 ][line width=0.75]      (0, 0) circle [x radius= 3.35, y radius= 3.35]   ;
\draw    (470.51,107.87) -- (431.44,107.69) ;
\draw [color={rgb, 255:red, 208; green, 2; blue, 27 }  ,draw opacity=1 ]   (431.44,107.69) ;
\draw [shift={(431.44,107.69)}, rotate = 0] [color={rgb, 255:red, 208; green, 2; blue, 27 }  ,draw opacity=1 ][fill={rgb, 255:red, 208; green, 2; blue, 27 }  ,fill opacity=1 ][line width=0.75]      (0, 0) circle [x radius= 3.35, y radius= 3.35]   ;
\draw [color={rgb, 255:red, 208; green, 2; blue, 27 }  ,draw opacity=1 ]   (470.51,107.87) ;
\draw [shift={(470.51,107.87)}, rotate = 0] [color={rgb, 255:red, 208; green, 2; blue, 27 }  ,draw opacity=1 ][fill={rgb, 255:red, 208; green, 2; blue, 27 }  ,fill opacity=1 ][line width=0.75]      (0, 0) circle [x radius= 3.35, y radius= 3.35]   ;
\draw   (511.31,81.16) .. controls (512.16,115.01) and (485.43,143.14) .. (451.61,143.98) .. controls (417.79,144.83) and (389.69,118.07) .. (388.84,84.21) .. controls (388,50.36) and (414.73,22.23) .. (448.55,21.39) .. controls (482.37,20.54) and (510.47,47.3) .. (511.31,81.16) -- cycle ;
\draw  [dash pattern={on 0.84pt off 2.51pt}]  (187.62,216.03) -- (186.75,255.09) ;
\draw [color={rgb, 255:red, 208; green, 2; blue, 27 }  ,draw opacity=1 ]   (186.75,255.09) ;
\draw [shift={(186.75,255.09)}, rotate = 0] [color={rgb, 255:red, 208; green, 2; blue, 27 }  ,draw opacity=1 ][fill={rgb, 255:red, 208; green, 2; blue, 27 }  ,fill opacity=1 ][line width=0.75]      (0, 0) circle [x radius= 3.35, y radius= 3.35]   ;
\draw [color={rgb, 255:red, 0; green, 0; blue, 0 }  ,draw opacity=1 ]   (187.62,216.03) ;
\draw [shift={(187.62,216.03)}, rotate = 0] [color={rgb, 255:red, 0; green, 0; blue, 0 }  ,draw opacity=1 ][fill={rgb, 255:red, 0; green, 0; blue, 0 }  ,fill opacity=1 ][line width=0.75]      (0, 0) circle [x radius= 3.35, y radius= 3.35]   ;
\draw  [dash pattern={on 0.84pt off 2.51pt}]  (210.62,215.86) -- (209.75,254.92) ;
\draw [color={rgb, 255:red, 208; green, 2; blue, 27 }  ,draw opacity=1 ]   (210.62,215.86) ;
\draw [shift={(210.62,215.86)}, rotate = 0] [color={rgb, 255:red, 208; green, 2; blue, 27 }  ,draw opacity=1 ][fill={rgb, 255:red, 208; green, 2; blue, 27 }  ,fill opacity=1 ][line width=0.75]      (0, 0) circle [x radius= 3.35, y radius= 3.35]   ;
\draw [color={rgb, 255:red, 0; green, 0; blue, 0 }  ,draw opacity=1 ]   (209.75,254.92) ;
\draw [shift={(209.75,254.92)}, rotate = 0] [color={rgb, 255:red, 0; green, 0; blue, 0 }  ,draw opacity=1 ][fill={rgb, 255:red, 0; green, 0; blue, 0 }  ,fill opacity=1 ][line width=0.75]      (0, 0) circle [x radius= 3.35, y radius= 3.35]   ;
\draw  [dash pattern={on 0.84pt off 2.51pt}]  (232.62,215.7) -- (231.75,254.76) ;
\draw [color={rgb, 255:red, 208; green, 2; blue, 27 }  ,draw opacity=1 ]   (231.75,254.76) ;
\draw [shift={(231.75,254.76)}, rotate = 0] [color={rgb, 255:red, 208; green, 2; blue, 27 }  ,draw opacity=1 ][fill={rgb, 255:red, 208; green, 2; blue, 27 }  ,fill opacity=1 ][line width=0.75]      (0, 0) circle [x radius= 3.35, y radius= 3.35]   ;
\draw [color={rgb, 255:red, 208; green, 2; blue, 27 }  ,draw opacity=1 ]   (232.62,215.7) ;
\draw [shift={(232.62,215.7)}, rotate = 0] [color={rgb, 255:red, 208; green, 2; blue, 27 }  ,draw opacity=1 ][fill={rgb, 255:red, 208; green, 2; blue, 27 }  ,fill opacity=1 ][line width=0.75]      (0, 0) circle [x radius= 3.35, y radius= 3.35]   ;
\draw   (220.63,177.43) .. controls (254.49,177.18) and (282.14,204.41) .. (282.39,238.23) .. controls (282.64,272.06) and (255.39,299.69) .. (221.53,299.94) .. controls (187.66,300.19) and (160.01,272.96) .. (159.76,239.13) .. controls (159.52,205.3) and (186.77,177.68) .. (220.63,177.43) -- cycle ;

\draw (396,230) node [anchor=north west][inner sep=0.75pt]  [font=\large] [align=left] {$\displaystyle \emptyset $};
\draw (251,231) node [anchor=north west][inner sep=0.75pt]  [font=\large] [align=left] {$\displaystyle \emptyset $};
\end{tikzpicture}
\]

However, if we apply the same method as we did in \ref{algorithm-coloring}, patching together the local solutions, we do not necessarily have a solution for the bigger graph. For example, consider the following set of local solutions that agree on the intersections.

\[
\tikzset{every picture/.style={line width=0.75pt}} 
\begin{tikzpicture}[x=0.65pt,y=0.65pt,yscale=-1,xscale=1]

\draw    (302.63,55.79) -- (303.53,109.58) ;
\draw [color={rgb, 255:red, 208; green, 2; blue, 27 }  ,draw opacity=1 ]   (303.53,109.58) ;
\draw [shift={(303.53,109.58)}, rotate = 0] [color={rgb, 255:red, 208; green, 2; blue, 27 }  ,draw opacity=1 ][fill={rgb, 255:red, 208; green, 2; blue, 27 }  ,fill opacity=1 ][line width=0.75]      (0, 0) circle [x radius= 3.35, y radius= 3.35]   ;
\draw [color={rgb, 255:red, 0; green, 0; blue, 0 }  ,draw opacity=1 ]   (302.63,55.79) ;
\draw [shift={(302.63,55.79)}, rotate = 0] [color={rgb, 255:red, 0; green, 0; blue, 0 }  ,draw opacity=1 ][fill={rgb, 255:red, 0; green, 0; blue, 0 }  ,fill opacity=1 ][line width=0.75]      (0, 0) circle [x radius= 3.35, y radius= 3.35]   ;
\draw   (323.82,237.68) .. controls (323.82,203.82) and (351.25,176.37) .. (385.08,176.37) .. controls (418.91,176.37) and (446.33,203.82) .. (446.33,237.68) .. controls (446.33,271.55) and (418.91,299) .. (385.08,299) .. controls (351.25,299) and (323.82,271.55) .. (323.82,237.68) -- cycle ;
\draw   (92.82,82.68) .. controls (92.82,48.82) and (120.25,21.37) .. (154.08,21.37) .. controls (187.91,21.37) and (215.33,48.82) .. (215.33,82.68) .. controls (215.33,116.55) and (187.91,144) .. (154.08,144) .. controls (120.25,144) and (92.82,116.55) .. (92.82,82.68) -- cycle ;
\draw   (302.63,21.43) .. controls (336.49,21.18) and (364.14,48.41) .. (364.39,82.23) .. controls (364.64,116.06) and (337.39,143.69) .. (303.53,143.94) .. controls (269.66,144.19) and (242.01,116.96) .. (241.76,83.13) .. controls (241.52,49.3) and (268.77,21.68) .. (302.63,21.43) -- cycle ;
\draw   (159.82,237.68) .. controls (159.82,203.82) and (187.25,176.37) .. (221.08,176.37) .. controls (254.91,176.37) and (282.33,203.82) .. (282.33,237.68) .. controls (282.33,271.55) and (254.91,299) .. (221.08,299) .. controls (187.25,299) and (159.82,271.55) .. (159.82,237.68) -- cycle ;
\draw    (176.33,148) -- (193.22,173.34) ;
\draw [shift={(194.33,175)}, rotate = 236.31] [color={rgb, 255:red, 0; green, 0; blue, 0 }  ][line width=0.75]    (10.93,-3.29) .. controls (6.95,-1.4) and (3.31,-0.3) .. (0,0) .. controls (3.31,0.3) and (6.95,1.4) .. (10.93,3.29)   ;
\draw    (336.33,144) -- (353.22,169.34) ;
\draw [shift={(354.33,171)}, rotate = 236.31] [color={rgb, 255:red, 0; green, 0; blue, 0 }  ][line width=0.75]    (10.93,-3.29) .. controls (6.95,-1.4) and (3.31,-0.3) .. (0,0) .. controls (3.31,0.3) and (6.95,1.4) .. (10.93,3.29)   ;
\draw    (270,145) -- (255.19,176.19) ;
\draw [shift={(254.33,178)}, rotate = 295.4] [color={rgb, 255:red, 0; green, 0; blue, 0 }  ][line width=0.75]    (10.93,-3.29) .. controls (6.95,-1.4) and (3.31,-0.3) .. (0,0) .. controls (3.31,0.3) and (6.95,1.4) .. (10.93,3.29)   ;
\draw    (428.33,147) -- (418.02,175.12) ;
\draw [shift={(417.33,177)}, rotate = 290.14] [color={rgb, 255:red, 0; green, 0; blue, 0 }  ][line width=0.75]    (10.93,-3.29) .. controls (6.95,-1.4) and (3.31,-0.3) .. (0,0) .. controls (3.31,0.3) and (6.95,1.4) .. (10.93,3.29)   ;
\draw    (118.51,55) -- (183.22,118.67) ;
\draw    (118.51,55) -- (185.33,55) ;
\draw [color={rgb, 255:red, 0; green, 0; blue, 0 }  ,draw opacity=1 ]   (185.33,55) ;
\draw [shift={(185.33,55)}, rotate = 0] [color={rgb, 255:red, 0; green, 0; blue, 0 }  ,draw opacity=1 ][fill={rgb, 255:red, 0; green, 0; blue, 0 }  ,fill opacity=1 ][line width=0.75]      (0, 0) circle [x radius= 3.35, y radius= 3.35]   ;
\draw [color={rgb, 255:red, 208; green, 2; blue, 27 }  ,draw opacity=1 ]   (183.22,118.67) ;
\draw [shift={(183.22,118.67)}, rotate = 0] [color={rgb, 255:red, 208; green, 2; blue, 27 }  ,draw opacity=1 ][fill={rgb, 255:red, 208; green, 2; blue, 27 }  ,fill opacity=1 ][line width=0.75]      (0, 0) circle [x radius= 3.35, y radius= 3.35]   ;
\draw [color={rgb, 255:red, 208; green, 2; blue, 27 }  ,draw opacity=1 ]   (118.51,55) ;
\draw [shift={(118.51,55)}, rotate = 0] [color={rgb, 255:red, 208; green, 2; blue, 27 }  ,draw opacity=1 ][fill={rgb, 255:red, 208; green, 2; blue, 27 }  ,fill opacity=1 ][line width=0.75]      (0, 0) circle [x radius= 3.35, y radius= 3.35]   ;
\draw   (511.31,81.16) .. controls (512.16,115.01) and (485.43,143.14) .. (451.61,143.98) .. controls (417.79,144.83) and (389.69,118.07) .. (388.84,84.21) .. controls (388,50.36) and (414.73,22.23) .. (448.55,21.39) .. controls (482.37,20.54) and (510.47,47.3) .. (511.31,81.16) -- cycle ;
\draw    (476.98,82.65) -- (423.18,82.72) ;
\draw [color={rgb, 255:red, 208; green, 2; blue, 27 }  ,draw opacity=1 ]   (476.98,82.65) ;
\draw [shift={(476.98,82.65)}, rotate = 0] [color={rgb, 255:red, 208; green, 2; blue, 27 }  ,draw opacity=1 ][fill={rgb, 255:red, 208; green, 2; blue, 27 }  ,fill opacity=1 ][line width=0.75]      (0, 0) circle [x radius= 3.35, y radius= 3.35]   ;
\draw [color={rgb, 255:red, 0; green, 0; blue, 0 }  ,draw opacity=1 ]   (423.18,82.72) ;
\draw [shift={(423.18,82.72)}, rotate = 0] [color={rgb, 255:red, 0; green, 0; blue, 0 }  ,draw opacity=1 ][fill={rgb, 255:red, 0; green, 0; blue, 0 }  ,fill opacity=1 ][line width=0.75]      (0, 0) circle [x radius= 3.35, y radius= 3.35]   ;
\draw [color={rgb, 255:red, 208; green, 2; blue, 27 }  ,draw opacity=1 ]   (221.53,264.58) ;
\draw [shift={(221.53,264.58)}, rotate = 0] [color={rgb, 255:red, 208; green, 2; blue, 27 }  ,draw opacity=1 ][fill={rgb, 255:red, 208; green, 2; blue, 27 }  ,fill opacity=1 ][line width=0.75]      (0, 0) circle [x radius= 3.35, y radius= 3.35]   ;
\draw [color={rgb, 255:red, 0; green, 0; blue, 0 }  ,draw opacity=1 ]   (220.63,210.79) ;
\draw [shift={(220.63,210.79)}, rotate = 0] [color={rgb, 255:red, 0; green, 0; blue, 0 }  ,draw opacity=1 ][fill={rgb, 255:red, 0; green, 0; blue, 0 }  ,fill opacity=1 ][line width=0.75]      (0, 0) circle [x radius= 3.35, y radius= 3.35]   ;

\draw (376,231) node [anchor=north west][inner sep=0.75pt]  [font=\large] [align=left] {$\displaystyle \emptyset $};
\end{tikzpicture}
\]
These do not patch together into a global solution because, although they hit all the edges, the vertex cover they form has size strictly greater than $2$.

\[
\tikzset{every picture/.style={line width=0.75pt}} 
\begin{tikzpicture}[x=0.65pt,y=0.65pt,yscale=-1,xscale=1]

\draw    (321.33,94) -- (415.33,94) ;
\draw    (321.33,94) -- (318.33,185) ;
\draw    (226.33,94) -- (321.33,94) ;
\draw [shift={(321.33,94)}, rotate = 0] [color={rgb, 255:red, 0; green, 0; blue, 0 }  ][fill={rgb, 255:red, 0; green, 0; blue, 0 }  ][line width=0.75]      (0, 0) circle [x radius= 3.35, y radius= 3.35]   ;
\draw    (226.33,94) -- (318.33,185) ;
\draw [shift={(226.33,94)}, rotate = 44.69] [color={rgb, 255:red, 0; green, 0; blue, 0 }  ][fill={rgb, 255:red, 0; green, 0; blue, 0 }  ][line width=0.75]      (0, 0) circle [x radius= 3.35, y radius= 3.35]   ;
\draw [color={rgb, 255:red, 208; green, 2; blue, 27 }  ,draw opacity=1 ]   (226.33,94) ;
\draw [shift={(226.33,94)}, rotate = 0] [color={rgb, 255:red, 208; green, 2; blue, 27 }  ,draw opacity=1 ][fill={rgb, 255:red, 208; green, 2; blue, 27 }  ,fill opacity=1 ][line width=0.75]      (0, 0) circle [x radius= 3.35, y radius= 3.35]   ;
\draw [color={rgb, 255:red, 208; green, 2; blue, 27 }  ,draw opacity=1 ]   (318.33,185) ;
\draw [shift={(318.33,185)}, rotate = 0] [color={rgb, 255:red, 208; green, 2; blue, 27 }  ,draw opacity=1 ][fill={rgb, 255:red, 208; green, 2; blue, 27 }  ,fill opacity=1 ][line width=0.75]      (0, 0) circle [x radius= 3.35, y radius= 3.35]   ;
\draw [color={rgb, 255:red, 208; green, 2; blue, 27 }  ,draw opacity=1 ]   (415.33,94) ;
\draw [shift={(415.33,94)}, rotate = 0] [color={rgb, 255:red, 208; green, 2; blue, 27 }  ,draw opacity=1 ][fill={rgb, 255:red, 208; green, 2; blue, 27 }  ,fill opacity=1 ][line width=0.75]      (0, 0) circle [x radius= 3.35, y radius= 3.35]   ;
\end{tikzpicture}
\]
}

\body{The naive algorithm we used for \textsc{GraphColoring} no longer applies for \textsc{VertexCover} and this is a well-known fact in algorithmics~\cite{cygan2015parameterized}. In this second problem, we found an obstruction: the size restriction of the vertex cover is local information that does not get carried over from the local parts to the global instance. Knowing these obstructions becomes important to be able to construct dynamic programming algorithms that \textit{do} work since, to do so, one needs to be able to account for and track such local obstructions. In the following section, we will show how to use category-theoretic tools to reason about all such obstructions. To ground our abstraction firmly in a practical example, we will consider \textsc{VertexCover} as a case study. However, we wish to emphasize from the beginning that all the methods we will employ generalize to any suitable computational problem.
}

\section{Presheaves and Čech Cohomology: \textsc{VertexCover} as a Case Study}

\body{
In this section, we will view \textsc{VertexCover} through a category-theoretic lens. To aid readers not familiar with category theory and as a warm-up for what follows in this paper, the constructions in this section are purposefully presented in more familiar set-theoretic terms. In Section~\ref{sec: abstract tools and more examples}, we will employ slicker categorical devices to exhibit more examples and, in doing so, we will also restate some of the constructions in this section in more abstract terms.
}

\subsection{Warmup: \textsc{VertexCover} as a Presheaf}

\begin{definition}
    Given a graph $G$, we say that a vertex subset $A \subseteq V(G)$ is a \define{vertex cover} of $G$ if $G - A$ is edgeless (equivalently, if for every edge $\alpha =\{u,v\} \in E(G)$, $u \in A$ or $v \in A$). We say that $A$ is a \textbf{minimum vertex cover of $G$} if, for every vertex cover $B \subseteq V(G)$, $|A| \leq |B|$.
\end{definition}

\body{
As far as vertex covers are concerned, the computational task at hand is that of finding a minimum vertex cover in a given input graph $G$. As is standard, this is stated as a decision problem as follows.
}

\vspace{0.4cm}

\fbox{\begin{minipage}{41em}
$\textsc{VertexCover}$

\textbf{Input:} A graph $G$ and an integer $k$.

\textbf{Task:} Decide if $G$ admits a vertex cover of cardinality at most $k$.
\end{minipage}}

\body{
The \define{vertex cover number} (i.e. the minimum size of a vertex cover in a given graph), is monotone under the subgraph relation, that is, the vertex cover number of a subgraph cannot be larger than that of its ambient graph. This suggests that certificates of \textsc{VertexCover} may be well-behaved with respect to monomorphisms\footnote{A \emph{monomorphism of graphs} is a homomorphism of graphs witnessing the inclusion of subgraph, i.e. injective on vertices and edges.}. Indeed, a vertex cover of a graph $G$ induces a vertex cover of any of its subgraphs. Here we point out that this is enough information to define a \textit{presheaf}: i.e. a functorial assignment of sets of certificates to each graph equipped with maps that relate the certificates of one graph to the those of another. To do so, we will first recall some basic notions from category theory. 
}

\begin{definition} \label{def: category}
    A \define{category} $\cat{C}$ consists of a collection $\cat{C}_0$ of objects and a collection $\cat{C}_1$ of morphisms (arrows) with the following operations:
    \begin{itemize}
        \item For each morphism $f$ of $\cat{C}_1$, we assign an object $a = dom(f)$ of $\cat{C}_0$ to be the domain of $f$ and an object $b = cod(f)$ of $\cat{C}_0$ to be the codomain of $f$; we use the notation $f \colon a \to b$.
        
        \item For each pair $(f,g)$ of morphisms of $\cat{C}_1$ such that $dom(g) = cod(f)$, we assign a morphism $g \circ f$ called composition of $f$ and $g$, where $dom(g \circ f) = dom(f)$ and $cod(g \circ f) = cod (g)$, that satisfy the
        
        \textit{Associative law}: given morphisms of $\cat{C}_1$
        \[\begin{tikzcd}
        a \arrow[r, "f"] & b \arrow[r, "g"] & c \arrow[r, "h"] & d
        \end{tikzcd}\]
        we have $(h \circ g) \circ f = h \circ (g \circ f)$.

        \item For each object $a$ of $\cat{C}_0$, we assign a morphism $id_a : a \to a$ of $\cat{C}_1$ that satisfy the 
        
        \textit{Identity law}: given morphisms of $\cat{C}_1$
        \[\begin{tikzcd}
        a \arrow[r, "f"] & b \arrow[r, "g"] & c
        \end{tikzcd}\]  
        we have
        \[ 1_b \circ f = f,~~~~~ g \circ 1_b = g. \]
    \end{itemize}
\end{definition}

\begin{definition}\label{def: preshaf}
    A \define{presheaf of sets} over a small category $\cat{C}$ is a \textit{contravariant functor} \[F \colon \cat{C}^{\op} \to \Set.\]
    Spelling this out, a presheaf of sets $F$ is an assignment of a \textit{set} $F(x)$ to each object $x \in \cat{C}$ and a \textit{function} $F(f) \colon F(y) \to F(x)$ to each morphism $f\colon x \to y$ in $\cat{C}$ such that the following two conditions hold: 
    \begin{enumerate}
        \item \textit{$F$ preserves composition} (i.e. if $x \xrightarrow{f} y$ and $y \xrightarrow{g} z$ are two morphisms in $\cat{C}$, then $F(g\circ f) = F(f) \circ F(g)$) and 
        \item \textit{$F$ preserves identities} (i.e. $F(\id_x) = \id_{F(x)}$ for all objects $x$).
    \end{enumerate}
    An element $A \in F(x)$ is said to be a \define{section of $F$ at $x$} and, accordingly, $F(x)$ is referred to as the \define{set of sections} at $x$. The functions $F(f)$ are called \define{restriction maps}. 
\end{definition}

\body{
Staying as concrete as possible and in an attempt to bridge graph-theoretic and category-theoretic language, throughout most of this article we will be restricting our attention to the posetal category\footnote{The reader knowledgeable in sheaf theory should be able to imagine how the story goes when one replaces $\sub(G)$ with another suitable \textit{site}.}~$\sub(G)$ where $G$ is some given graph. This category, defined formally below, has subgraphs of $G$ as objects and injective graph homomorphisms as arrows. 
}

\begin{definition}
    Given a graph $G$, the category $\sub(G)$ has subgraphs of $G$ as its objects. We will say that there is an arrow $f \colon H \to K$ between two subgraphs of $G$ if $f$ is a homomorphism of graphs and the following diagram is commutative. 

    \[\begin{tikzcd}[sep=small]
    H \arrow[rr, "f"] \arrow[rd, hook] &   & K \arrow[ld, hook] \\
                                       & G &                    
    \end{tikzcd} \]
\end{definition}

\body{
It is not difficult to see that if $f$ makes the above diagram commute, then $f$ is injective. So, to give an arrow between $H,K$ subgraphs of $G$ is equivalent to say $H \subseteq K$ -- in particular, there is at most one such map.
}

\begin{figure}
    \centering
    \include{figures/vertex-cover-presheaf}
    \caption{The presheaf with domain $\sub(K_2)$ associated to the problem of deciding if a graph admits a vertex cover of cardinality at most $1$.}
    \label{fig: vertex cover presheaf}
\end{figure}

\body{
A sweeping conceptual point we wish to make here is that, given any computational problem $F$, it is fruitful to consider morphisms with respect to which the following assignment is functorial: \[\text{instance } X \mapsto \text{ all certificates of } F \text{ on } X.\]
This is the case for \textsc{VertexCover}: there is a presheaf $\Va \colon \sub(G)^{\op} \to \Set$ which sends each subgraph of $G$ to the set of all of its vertex covers. This yields the assignment
\begin{align*}
 \mathcal{V} \colon \sub(G)^{\cat{op}} &\longrightarrow  \Set \\
 H \quad \quad &\longmapsto  \{A \subseteq V(H) : A ~\text{is a vertex cover of }H\}
\end{align*}
equipped with the following restriction maps: given $f \colon H \hookrightarrow K$, we put
\begin{align*}
 \Va(f) \colon \Va(K) &\longrightarrow  \Va(H) \\
 A \quad &\longmapsto  A\restrict_H := \{v \in A: v \in V(H)\}.
\end{align*}
This construction is well defined and indeed yields a presheaf (Lemma~\ref{vertex cover functor is a presheaf} below). Notice that, for now, we put no restrictions on the size of the vertex covers; this will be done later (see Figure~\ref{fig: vertex cover presheaf}) when we consider the presheaf $\Va_{\leq k}$ (e.g. in Proposition~\ref{prop: vert cover < k is not sheaf}).}

\begin{lemma}\label{vertex cover functor is a presheaf}
    $\Va$ is a presheaf.
\end{lemma}
\begin{proof}

We need to prove that given $f \colon H \hookrightarrow K$ and $g \colon K \hookrightarrow J$, one has  
$\mathcal{V}(g \circ f) = \mathcal{V}(f) \circ \mathcal{V}(g)$ 
and that
$\mathcal{V}(id_H) = id_{\mathcal{V}(H)}$.

For the first one, given $A$ a vertex cover of $J$, we have 
$\mathcal{V}(g \circ f)(A) = A\restrict_H$ and $[\mathcal{V}(f) \circ \mathcal{V}(g)](A) = (A\restrict_{K})\restrict_H$. Note that 
$$(A\restrict_{K})\restrict_H = \{v \in V(H): v \in A\restrict_{K}\} = \{v \in V(H) : v \in K \wedge v \in A \}.$$ 
Since $H \subseteq K$, then 
$$\{v \in V(H) : v \in V(K) \wedge v \in A \} = \{v \in V(H) : v \in A\} = A\restrict_H.$$
Hence, we have that $\mathcal{V}(g \circ f) = \mathcal{V}(f) \circ \mathcal{V}(g)$.

Finally, we also have that $\mathcal{V}(id_H) = id_{\mathcal{V}(H)}$ simply because given $A$ a vertex cover of $H$, $A \subseteq V(H)$, and therefore
$$A\restrict_H = \{v \in V(H) : v \in A\} = A.$$
\end{proof}

\body{
The restriction maps of a presheaf allow us to systematically track how global certificates can be converted into local ones. Notice, however, that, since we have \textit{not placed any restriction} on the size of the vertex cover in the presheaf above, one can actually go the other way: given any collection of vertex covers of the subgraphs, if these vertex covers agree on intersections, then they can be patched together \textit{uniquely} to form a global vertex cover. Presheaves that satisfy this condition are known as \textit{sheaves} (defined below). We will show (Lemma~\ref{prop: V is a sheaf}) that $\Va$ is an example of such a construction; another example is the $n$-coloring functor (as we alluded to earlier in \ref{algorithm-coloring}, without proof).  
}

\begin{definition} \label{matching family}
    Fix a presheaf $F \colon \sub(G)^{\op} \to \Set$ and a collection of subgraphs $\Ua := (H_i)_{i \in I}$ which cover a subgraph $H$ of $G$ (i.e. such that $\bigcup_i H_i = H$). A \define{matching family} of $F$ on $\Ua$ is  family $\{A_i \in F(H_i)\}_{i \in I}$ of local sections that agree on the intersections, that is, for all $i,j \in I$ we have 
    $$A_i\restrict_{H_i \cap H_j} = A_j\restrict_{H_i \cap H_j}.$$ 
\end{definition}

\body{
Notice how every global section $A \in F(H)$ induces a matching family on any cover $(H_i)_{i \in I}$ of $H$, simply by defining $A_i := A \restrict_{H_i}$.
A presheaf is a sheaf when all matching families arise uniquely in such a way:
}

\begin{definition}\label{sheaf set}
    We say that a presheaf $F\colon \sub(G)^{\op} \to \cat{Set}$ is a \define{sheaf} if it satisfies the following \textit{sheaf condition}:
    \begin{quote}
        For any matching family $\{A_i \in F(H_i)\}_{i \in I}$ of $F$ on a cover $(H_i)_{i \in I}$ of any given $H \in \sub(G)$, there exists a unique global section $A \in F(H)$ that agrees with the local ones, i.e., $A\restrict_{H_i} = A_i$ for all $i \in I$.
    \end{quote}
    
    We call $A$ the \define{amalgamation} of the matching family $\{A_i\}_{i \in I}$.
\end{definition}

\begin{proposition}\label{prop: V is a sheaf}
    $\Va$ is a sheaf.
\end{proposition}
\begin{proof}
Suppose that we have $H \in \sub(G)$ and a cover of graphs $\{H_i\}_{i \in I}$ for $H$. Moreover, suppose that we have a family $\{A_i\}_{i \in I}$ of vertex covers $A_i \in \mathcal{V}(H_i)$ that agree on the intersections, that is, for all $i,j \in I$ we have 
$$A_i\restrict_{H_i \cap H_j} = A_j\restrict_{H_i \cap H_j}.$$ Let's see that $\bigcup_{i \in I} A_i$ is the unique vertex cover of $H$ that agrees with the other vertex covers.

First, we note that $A = \bigcup_{i \in I} A_i$ is a vertex cover of $H$. Suppose not, then there exists $\alpha$ an edge in $H - A$. Since $H = \bigcup_{i \in I} H_i$, there exists a $i \in I$ such that $\alpha \in E(H_i)$. Now, because the vertices of $\alpha$ are not in $A$ and $A$ is the union $\bigcup_{i \in I} A_i$, we have that these vertices are also not in $A_i$. Hence $\alpha \in E(H_i - A_i)$, contradicting the fact that $A_i$ is a vertex cover of $H_i$.

The union $A = \bigcup_{i \in I} A_i$ satisfies $A\restrict_{H_i} = A_i$ for all $i \in I$, so let's show that it is the unique vertex cover with this property. Suppose $B$ is a vertex cover of $H$ satisfying the property. Since $A_i = B\restrict_{H_i}$ and $B\restrict_{H_i} \subseteq B$, for all $i$, we have $A = \bigcup_{i \in I} A_i \subseteq B$. Now, if $v \in B \subseteq V(H)$, since $H = \bigcup_{i \in I} H_i$, there exists a $i \in I$ such that $v \in V(H_i)$. We then have $v \in A_i$, because $B\restrict_{H_i} = A_i$. Therefore, $B \subseteq \bigcup_{i \in I} A_i$, and we proved that $B = \bigcup_{i \in I} A_i = A$.
\end{proof}

\body{
To view presheaves (or sheaves for that matter) as computational problems, one defines the following general computational task. 
\vspace{0.4cm}

\defproblem{$F$-Decision}{An object $x \in \cat{C}$ where $\cat{C}$ is the domain of the fixed presheaf $F \colon \cat{C}^{\op} \to \Set$.} {Decide if $x$ has a non-empty set of certificates;  in other words, output ``yes'' if $F(x)$ is non-empty and ``no'' otherwise.}
}

\body{
At this point the reader is going to protest that the sheaf $\Va$ we just described does not encode an interesting decision problem, since all instances of $\Va$-\textsc{Decision} are yes-instances (the fact that we have placed no restriction on the sizes of the vertex covers implies that the entire vertex set is always a certificate). This is indeed the case and the reason we have taken the time to study $\Va$ is that it will appear naturally in connection with the following presheaf that does indeed encode \textsc{VertexCover}. 
} 

\body{Given a $k \in \mathbb{N}$, we consider

\[\begin{array}{ccccc}
    & \mathcal{V}_{\leq k} \colon & \sub(G)^{\cat{op}} & \longrightarrow         &  \Set\\
    \\
    &         & H              & \longmapsto &  \{A \subseteq V(H): A ~\text{is a vertex cover of }H\text{ and } |A|\leq k\}
    \end{array}\]

and we define the action on the arrows as before, because when we restrict a vertex cover with size at most $k$, we have a vertex cover with size at most $k$ (see Figure~\ref{fig: vertex cover presheaf}).}

\body{Now the same calculations as before show that $\mathcal{V}_{\leq k}$ is a functor. But, in this case, it is not necessarily a sheaf. This is because when we consider the union of the vertex covers of the smaller graphs (and we just saw in Proposition~\ref{prop: V is a sheaf} that this is the only possible vertex cover for the bigger graph), it can end up having size greater than $k$, as we will see next. In some sense, being reminiscent of the presheaf of bounded functions on a topological space (which famously fails to be a sheaf~\cite{rosiak-book}) the following is a prototypical failure of the sheaf condition.}

\begin{proposition} \label{prop: vert cover < k is not sheaf}
    $\Va_{\leq k}$ is not a sheaf.
\end{proposition}

\begin{proof}
Let's consider the following example, where we are looking for vertex covers of size one in a two-vertex complete graph, i.e., we set $k=1$ and $G = K_2$ (the idea can be adapted to any $k > 1$). 

\[
\tikzset{every picture/.style={line width=0.75pt}} 
\begin{tikzpicture}[x=0.75pt,y=0.75pt,yscale=-1,xscale=1]

\draw   (262,52.8) .. controls (262,40.13) and (295.94,29.85) .. (337.8,29.85) .. controls (379.66,29.85) and (413.6,40.13) .. (413.6,52.8) .. controls (413.6,65.47) and (379.66,75.75) .. (337.8,75.75) .. controls (295.94,75.75) and (262,65.47) .. (262,52.8) -- cycle ;
\draw   (281,131.3) .. controls (281,120.64) and (308,112) .. (341.3,112) .. controls (374.6,112) and (401.6,120.64) .. (401.6,131.3) .. controls (401.6,141.96) and (374.6,150.6) .. (341.3,150.6) .. controls (308,150.6) and (281,141.96) .. (281,131.3) -- cycle ;
\draw   (260,190.8) .. controls (260,182.07) and (269.54,175) .. (281.3,175) .. controls (293.06,175) and (302.6,182.07) .. (302.6,190.8) .. controls (302.6,199.53) and (293.06,206.6) .. (281.3,206.6) .. controls (269.54,206.6) and (260,199.53) .. (260,190.8) -- cycle ;
\draw   (319,238.74) .. controls (319,226.73) and (330.55,217) .. (344.8,217) .. controls (359.05,217) and (370.6,226.73) .. (370.6,238.74) .. controls (370.6,250.75) and (359.05,260.48) .. (344.8,260.48) .. controls (330.55,260.48) and (319,250.75) .. (319,238.74) -- cycle ;
\draw    (302,53) -- (373.6,52.6) ;
\draw [shift={(373.6,52.6)}, rotate = 359.68] [color={rgb, 255:red, 0; green, 0; blue, 0 }  ][fill={rgb, 255:red, 0; green, 0; blue, 0 }  ][line width=0.75]      (0, 0) circle [x radius= 3.35, y radius= 3.35]   ;
\draw [shift={(302,53)}, rotate = 359.68] [color={rgb, 255:red, 0; green, 0; blue, 0 }  ][fill={rgb, 255:red, 0; green, 0; blue, 0 }  ][line width=0.75]      (0, 0) circle [x radius= 3.35, y radius= 3.35]   ;
\draw    (304,133) ;
\draw [shift={(304,133)}, rotate = 0] [color={rgb, 255:red, 0; green, 0; blue, 0 }  ][fill={rgb, 255:red, 0; green, 0; blue, 0 }  ][line width=0.75]      (0, 0) circle [x radius= 3.35, y radius= 3.35]   ;
\draw    (382.6,132.48) ;
\draw [shift={(382.6,132.48)}, rotate = 0] [color={rgb, 255:red, 0; green, 0; blue, 0 }  ][fill={rgb, 255:red, 0; green, 0; blue, 0 }  ][line width=0.75]      (0, 0) circle [x radius= 3.35, y radius= 3.35]   ;
\draw    (281.3,190.8) ;
\draw [shift={(281.3,190.8)}, rotate = 0] [color={rgb, 255:red, 0; green, 0; blue, 0 }  ][fill={rgb, 255:red, 0; green, 0; blue, 0 }  ][line width=0.75]      (0, 0) circle [x radius= 3.35, y radius= 3.35]   ;
\draw    (412.3,191.8) ;
\draw [shift={(412.3,191.8)}, rotate = 0] [color={rgb, 255:red, 0; green, 0; blue, 0 }  ][fill={rgb, 255:red, 0; green, 0; blue, 0 }  ][line width=0.75]      (0, 0) circle [x radius= 3.35, y radius= 3.35]   ;
\draw    (299.6,171.6) -- (314.26,155.37) ;
\draw [shift={(315.6,153.88)}, rotate = 132.09] [color={rgb, 255:red, 0; green, 0; blue, 0 }  ][line width=0.75]    (10.93,-3.29) .. controls (6.95,-1.4) and (3.31,-0.3) .. (0,0) .. controls (3.31,0.3) and (6.95,1.4) .. (10.93,3.29)   ;
\draw [shift={(299.6,171.6)}, rotate = 312.09] [color={rgb, 255:red, 0; green, 0; blue, 0 }  ][line width=0.75]      (0,-11.18) .. controls (-3.09,-11.18) and (-5.59,-8.68) .. (-5.59,-5.59) .. controls (-5.59,-2.5) and (-3.09,0) .. (0,0) ;
\draw    (391.6,173.6) -- (387.21,167.78) -- (375.89,154.41) ;
\draw [shift={(374.6,152.88)}, rotate = 49.76] [color={rgb, 255:red, 0; green, 0; blue, 0 }  ][line width=0.75]    (10.93,-3.29) .. controls (6.95,-1.4) and (3.31,-0.3) .. (0,0) .. controls (3.31,0.3) and (6.95,1.4) .. (10.93,3.29)   ;
\draw [shift={(391.6,173.6)}, rotate = 232.94] [color={rgb, 255:red, 0; green, 0; blue, 0 }  ][line width=0.75]      (0,0) .. controls (-3.09,0) and (-5.59,2.5) .. (-5.59,5.59) .. controls (-5.59,8.68) and (-3.09,11.18) .. (0,11.18) ;
\draw    (380.6,226.6) -- (394.35,209.44) ;
\draw [shift={(395.6,207.88)}, rotate = 128.71] [color={rgb, 255:red, 0; green, 0; blue, 0 }  ][line width=0.75]    (10.93,-3.29) .. controls (6.95,-1.4) and (3.31,-0.3) .. (0,0) .. controls (3.31,0.3) and (6.95,1.4) .. (10.93,3.29)   ;
\draw [shift={(380.6,226.6)}, rotate = 308.71] [color={rgb, 255:red, 0; green, 0; blue, 0 }  ][line width=0.75]      (0,-11.18) .. controls (-3.09,-11.18) and (-5.59,-8.68) .. (-5.59,-5.59) .. controls (-5.59,-2.5) and (-3.09,0) .. (0,0) ;
\draw    (309.6,225.6) -- (293.94,208.36) ;
\draw [shift={(292.6,206.88)}, rotate = 47.75] [color={rgb, 255:red, 0; green, 0; blue, 0 }  ][line width=0.75]    (10.93,-3.29) .. controls (6.95,-1.4) and (3.31,-0.3) .. (0,0) .. controls (3.31,0.3) and (6.95,1.4) .. (10.93,3.29)   ;
\draw [shift={(309.6,225.6)}, rotate = 227.75] [color={rgb, 255:red, 0; green, 0; blue, 0 }  ][line width=0.75]      (0,0) .. controls (-3.09,0) and (-5.59,2.5) .. (-5.59,5.59) .. controls (-5.59,8.68) and (-3.09,11.18) .. (0,11.18) ;
\draw    (347.6,97.6) -- (346.41,91.2) -- (348.18,82.84) ;
\draw [shift={(348.6,80.88)}, rotate = 102.01] [color={rgb, 255:red, 0; green, 0; blue, 0 }  ][line width=0.75]    (10.93,-3.29) .. controls (6.95,-1.4) and (3.31,-0.3) .. (0,0) .. controls (3.31,0.3) and (6.95,1.4) .. (10.93,3.29)   ;
\draw [shift={(347.6,97.6)}, rotate = 259.43] [color={rgb, 255:red, 0; green, 0; blue, 0 }  ][line width=0.75]      (0,-11.18) .. controls (-3.09,-11.18) and (-5.59,-8.68) .. (-5.59,-5.59) .. controls (-5.59,-2.5) and (-3.09,0) .. (0,0) ;
\draw   (391,191.8) .. controls (391,183.07) and (400.54,176) .. (412.3,176) .. controls (424.06,176) and (433.6,183.07) .. (433.6,191.8) .. controls (433.6,200.53) and (424.06,207.6) .. (412.3,207.6) .. controls (400.54,207.6) and (391,200.53) .. (391,191.8) -- cycle ;

\draw (279,41) node [anchor=north west][inner sep=0.75pt]   [align=left] {$\displaystyle v_{1}$};
\draw (382,42) node [anchor=north west][inner sep=0.75pt]   [align=left] {$\displaystyle v_{2}$};
\draw (234,42) node [anchor=north west][inner sep=0.75pt]   [align=left] {$\displaystyle G$};
\draw (412,116) node [anchor=north west][inner sep=0.75pt]   [align=left] {$\displaystyle H_{3}$};
\draw (224,187) node [anchor=north west][inner sep=0.75pt]   [align=left] {$\displaystyle H_{1}$};
\draw (447,186) node [anchor=north west][inner sep=0.75pt]   [align=left] {$\displaystyle H_{2}$};
\draw (338.29,230.2) node [anchor=north west][inner sep=0.75pt]   [align=left] {$\displaystyle \emptyset $};
\end{tikzpicture}
\]

Consider the cover $\{\emptyset, H_1, H_2\}$ of $H_3$ and consider the matching family 
$$\emptyset \in \Va_{\leq 1}(\emptyset), \{v_1\} \in \Va_{\leq 1}(H_1), \{v_2\} \in \Va_{\leq 1}(H_2).$$
This family does not have an amalgamation in $\mathcal{V}_{\leq 1}(H_3) = \{\emptyset,\{v_1\},\{v_2\}\}$. In fact, as we saw in the proof of Proposition~\ref{prop: V is a sheaf}, the only possible amalgamation is the union of the local solutions, but in this case the union is $\{v_1,v_2\}$, which is not in $\Va_{\leq 1}(H_3)$.
\end{proof}

\body{Although $\Va_{\leq k}$ is not a sheaf, it is somehow halfway there. In the definition of the sheaf condition, the amalgamation has to satisfy two things: existence and uniqueness. In the case of $\Va_{\leq k}$, we have the latter: when the amalgamation of matching family exists, it is unique, so this presheaf fails to be a sheaf only with respect to the ``existence'' of amalgamations. Since a presheaf can fail to be a sheaf in these two ways, we present the following terminology that encodes what happens when a presheaf satisfies either uniqueness or existence.

\begin{definition} \label{sep-lav-set} We say that a presheaf $F \colon \sub(G)^{\cat{op}} \to \Set$ is
    \begin{itemize}
        \item[(a)] \textbf{separated} if given $H \in \sub(G)$ and a cover $\{H_i\}_{i \in I}$ for $H$, if $A,B \in F(H)$ are such that $A \restrict _{H_{i}} = B\restrict_{H_{i}}$ for all $i \in I$, then $A = B$. 
        \item[(b)] \textbf{lavish}\footnote{What we call \textit{lavish} presheaves are often referred to as \textit{epi-presheaves}. We find this terminology confusing since \textit{flasque} presheaves (i.e. presheaves whose restriction maps are epic) could also be deserving of the name ``epi-presheaf''. To mitigate any source of confusion, we stray from any use of this term by sticking to ``lavish'' and ``flasque''. Furthermore, and this is part of the moral of the present paper, lavish presheaves deserve their own name since they are in many respects just as important as their separated counterpart.} if given $H \in \sub(G)$, a cover $\{H_i\}_{i \in I}$ for $H$ and a matching family $\{A_i \in F(H_i)\}_{i \in I}$, there exists $A \in F(H)$ such that $A \restrict_{H_{i}} = A_i$ for all $i \in I$.
    \end{itemize}
\end{definition}

\begin{proposition} \label{prop:V-separated}
    $\Va_{\leq k}$ is a separated presheaf.
\end{proposition}

\begin{proof}
    Given $H \in \sub(G)$ and a cover $\{H_i\}_{i \in I}$ for $H$, suppose $A,B \in \Va_{\leq k}(H)$ are such that $A \restrict _{H_{i}} = B\restrict_{H_{i}}$ for all $i \in I$. If $v \in A \subseteq V(H)$, since $H = \bigcup_{i \in I}H_i$, there exists a $i \in I$ such that $v \in V(H_i)$. Therefore, $v \in A \restrict _{H_{i}}$, and since $A \restrict _{H_{i}} = B\restrict_{H_{i}}$ this implies that $v \in B \restrict _{H_{i}}$, ie, $v \in B$. The inclusion $B \subseteq A$ is similar.
\end{proof}

\body{There is a well-known construction called \textit{sheafification}~\cite{rosiak-book} that allows one to construct a sheaf from any given presheaf. As was shown recently~\cite{deciding-sheaves}, if a decision problem can be encoded as a sheaf, then it admits a fast (i.e. $\fpt$-time) dynamic programming algorithm. As we will see, although it is of great use in geometry, the sheafification construction is not immediately useful for algorithmic considerations since, for example, when it is applied to $\Va_{\leq k}$, it yields the sheaf $\Va$ (which, we have already seen, does not encode an interesting computational problem). However, as we will see in Section~\ref{sec: abstract tools and more examples}, sheafification can be conceptually useful when trying to understand \textit{obstructions to algorithmic compositionality}, in other words, obstructions to begin a sheaf. To that end, in the following subsection we will first revisit the definition of a sheaf in a more category-theoretic way and recall the definition of the zeroth cohomology group.}

\subsection{Some Background on Cohomology: Obstructions to being a sheaf}

\body{
In general, what it means for a presheaf to be a sheaf (resp. a separated or lavish presheaf) can be stated in terms of an \textit{equalizer diagram}. In the interest of being self-contained, we include the definition below.
}

\begin{definition}
    Given two morphisms $f,g: x \to y$ in a category $\cat{C}$, we say that an object $z \in \cat{C}$ together with a morphism $h: z \to x$ is an \textbf{equalizer of $f$ and $g$} if
    \begin{itemize}
        \item $f \circ h = g \circ h$;
        \item Whenever $j: w \to x$ satisfies $f \circ j = g \circ j$, there exists a unique morphism $k: w \to z$ such that $h \circ k = j$.
        \[
    \begin{tikzcd}
    z \arrow[r, "h"]                                    & x \arrow[r, "f", shift left] \arrow[r, "g"', shift right] & y \\
    w \arrow[ru, "j"'] \arrow[u, "\exists! ~k", dashed] &                                                           &  
    \end{tikzcd}
    \]
    \end{itemize}
\end{definition}

\body{Equalizers can be thought as categorical ways of encoding equalitites. In $\Set$, we have that the equalizer object of two functions $f,g: A \to B$ is the subset $E = \{a \in A : f(a) = g(a)\}$ of $A$ on which $f$ and $g$ agree. 
}

\body{With this definition in mind, we can now recall a more categorical framing of the 
properties described in Definitions~\ref{sheaf set} and~\ref{sep-lav-set}.}

\begin{definition} \label{sheaf cat}
    Let $F \colon \sub(G)^{\op} \to \Set$ be a presheaf. For any object $H \in \sub(G)$ and any cover
    $$ \Ua = (H_i \xhookrightarrow{f_i} H)_{i \in I} $$
    of $H$, the \textbf{matching family of $F$ on $\Ua$} are defined to be the following equalizer in $\Set$ (which we will denote as $\Match(\Ua,F$))

    \[\begin{tikzcd}
	{\cat{Eq}(q_{1,0}, q_{1,1})} && {\prod_{i \in I}F(H_i)} && {\prod_{i,j \in I}F(H_i \cap H_j)}
	\arrow["{q_{1,0}}", shift left=2, from=1-3, to=1-5]
	\arrow["{q_{1,1}}"', shift right=2, from=1-3, to=1-5]
	\arrow[hook, from=1-1, to=1-3]
\end{tikzcd}\]
where $q_{1,0}$ is given for each $i,j \in I$ by the restriction map $F(H_i) \to F(H_i \cap H_j)$, while $q_{1,1}$ by the restriction $F(H_j) \to F(H_i \cap H_j)$.

    The maps $(F(H) \xrightarrow{F(f_i)} F(H_i))_{i \in I}$ assemble into a morphism

    \[\begin{array}{ccccc}
    & \delta^{-1} \colon & F(H) & \longrightarrow         &  \prod_{i \in I} F(H_i)\\
    \\
    &         & A              & \longmapsto &  (F(f_i)(A))_{i \in I}
    \end{array}\]
    
    equalizing $q_{1,0}$ and $q_{1,1}$ (where the superscript in $\delta^{-1}$ denotes an index, not an inverse map). Thus, by the universal property of equalizers, one has a unique map $\xi$ making the following commute.
    \[
\begin{tikzcd}
F(H) \arrow[rr, "\delta^{-1}"] \arrow[rd, "\exists!\xi"', dashed] &                                        & \prod_{i \in I}F(H_i) \arrow[rr,shift left=.75ex,"q_{1,0}"]
\arrow[rr,shift right=.75ex,swap,"q_{1,1}"] &  & {\prod_{i,j \in I}F(H_i \cap H_j)} \\
                                                                  & {Eq(q_{1,0},q_{1,1})} \arrow[ru, hook] &                                          &  &                                         
\end{tikzcd}\]
    If, for all objects $H$ and all covers $\Ua$ thereof, the arrow $\xi$ is
    \begin{itemize}
        \item a monomorphism, then one says that $F$ is \textbf{separated};
        \item an epimorphism, then one says that $F$ is \textbf{lavish};
        \item an isomorphism, then one says that $F$ is \textbf{sheaf}.
    \end{itemize}
    
    \end{definition}

\begin{example}
    Taking $F$ in the definition above to be $\mathcal{V}_{\leq k}$ and setting $\Ua = \{H_1, H_2, H_3\}$ to be a cover of some graph $H \subseteq G$, we have that
\[\begin{array}{ccccc}
    & q_{1,0} \colon & \Va_{\leq k}(H_1) \times \Va_{\leq k}(H_2) \times \Va_{\leq k}(H_3) & \longrightarrow         &  \Va_{\leq k}(H_1 \cap H_2) \times \Va_{\leq k}(H_1 \cap H_3) \times \Va_{\leq k}(H_2 \cap H_3)\\
    \\
    &         & (A_1, A_2, A_3)              & \longmapsto &  (A_1 \restrict_{H_{1} \cap H_{2}}, \: A_{1} \restrict_{H_{1} \cap H_{3}}, \:  A_2\restrict_{H_{2} \cap H_{3}})
    \end{array}\]
when we restrict to the first index, and 
\[\begin{array}{ccccc}
    & q_{1,1} \colon & \Va_{\leq k}(H_1) \times \Va_{\leq k}(H_2) \times \Va_{\leq k}(H_3) & \longrightarrow         &  \Va_{\leq k}(H_1 \cap H_2) \times \Va_{\leq k}(H_1 \cap H_3) \times \Va_{\leq k}(H_2 \cap H_3)\\
    \\
    &         & (A_1, A_2, A_3)              & \longmapsto &  (A_2 \restrict_{H_{1} \cap H_{2}}, \: A_{3} \restrict_{H_{1} \cap H_{3}}, \:  A_3\restrict_{H_{2} \cap H_{3}})
    \end{array}\]
when we restrict to the second one.
 Now, the set of triples $(A_1,A_2,A_3)$ that have the same image under $q_{0}$ and $q_{1}$ are exactly the families of solutions that agree on the intersections, that is, the matching families as we defined in \ref{sheaf set}. Also, $\xi$ has the same action of $\delta^{-1}$, i.e., the restriction of solutions; thus, to say that $\xi$ is mono is to say that two solutions of $H$ that agree on each restriction are the same and to say that $\xi$ is epi is to say that for each matching family of local solutions there exists at least a solution of $H$ whose  restriction to each subgraph corresponds to the local solutions of the matching family. 
\end{example}

\body{Notice, as it is completely standard, that if we consider $\cat{Ab}$, the category of Abelian groups instead of $\Set$ in the definition of sheaf, we can rewrite the set where the functions agree to be the kernel of the difference $q_{1,1} - q_{1,0}$. This approach will be useful to us as we wish to apply standard constructions in cohomology. In order to talk about Čech cohomology, we first recall the notion of Čech nerve.}

\definition{
Consider $X \in \sub(G)$ and a cover $\Ua = (H_i \xhookrightarrow{f_i} X)_{i\in I}$. The \define{Čech nerve} of the cover $\Ua$ -- denoted $N(\Ua)$ --  is the following simplicial object in $\sub(G)$ (whose degeneracy maps are left unnamed for clarity): 
\[
    \begin{tikzcd}[ampersand replacement=\&]
    	X \&\& {\coprodcup_{i \in I} H_i} \&\& {\coprodcup_{i,j \in I} H_{i} \cap H_{j}} \&\& {\coprodcup_{i,j,k \in I} H_{i} \cap H_j \cap H_k} \& \dots
    	\arrow["{d_0}"', from=1-3, to=1-1]
    	\arrow["{d_{1,1}}", shift left=3, from=1-5, to=1-3]
    	\arrow["{d_{1,0}}"', shift right=3, from=1-5, to=1-3]
    	\arrow["{d_{2,2}}", shift left=5, from=1-7, to=1-5]
    	\arrow["{d_{2,0}}"', shift right=5, from=1-7, to=1-5]
    	\arrow["{d_{2,1}}"{description}, from=1-7, to=1-5]
    	\arrow[from=1-3, to=1-5]
    	\arrow[shift left=3, from=1-5, to=1-7]
    	\arrow[shift right=3, from=1-5, to=1-7]
    \end{tikzcd}
\]
Here the parallel morphisms are defined the usual way (e.g., implicitly assuming an ordering on the index set $I$, one has that $d_{2, 1} \colon H_{i} \cap H_j \cap H_k \to H_{i} \cap H_k$).
}
\body{
Now fix a presheaf $F: \sub(G)^{\sf op} \to \cat{Ab}$. Taking the image of a Čech nerve under $F$ induces the following cochain complex (note that the domain of $\delta^{-1}_{\Ua}$ is $F(X)$ -- i.e. the global sections -- instead of $0$; \textit{this differs from the standard treatment}\footnote{The reader familiar with standard treatments of Čech cohomology might expect $H^0(X,F)$ to be the global sections of $F$ whenever $F$ is a sheaf. However, under the constrution we are considering here, since the domain of $\delta^{-1}_{\Ua}$ is $F(H)$, $H^0(X,F)$ measures the difference between matching families ($\ker \delta^0$) and the global sections: thus it is \textbf{trivial} if $F$ is a sheaf. Note that there are no differences between our approach and the standard one for higher cohomology objects.})
\[\begin{tikzcd}
	{F(X)} && {\prod_i F(H_i)} && {\prod_{i,j} F (H_{i} \cap H_j)} && {\prod_{i,j,k} F(H_i \cap H_j \cap H_k)} & \dots
	\arrow["{\delta^{-1}_{\Ua}}", from=1-1, to=1-3]
	\arrow["{\delta^{0}_{\Ua}}", from=1-3, to=1-5]
	\arrow["{\delta^{1}_{\Ua}}", from=1-5, to=1-7]
\end{tikzcd}\]
where the coboundary maps are given by 
\[
    \delta^{-1}_{\Ua} \defeq F(d_0) \quad \text{ and } \quad
    \delta^n_{\Ua}    \defeq \sum_{i = 0}^n (-1)^i F(d_{n+1, i}).
\]
Since $F$ takes values in $\cat{Ab}$, one can define the $n$-th Čech cohomology of the presheaf $F$ on the object $H$ with cover $\Ua$ as the quotient of $\ker \delta_{\Ua}^n$ by $\imageObj \delta_{\Ua}^{n-1}$ as in the Definition~\ref{def:cohomology-cover-specific} below. 
}

\begin{definition}\label{def:cohomology-cover-specific}
The \textbf{$n$-th Čech cohomology} of a presheaf $F\colon \sub(G)^{\op} \to \cat{Set}$ on an object $X$ with cover $\Ua$ is defined as
\[H^n(X, \Ua, F) := \ker(\delta^n) / \im(\delta^{n-1}).\]
\end{definition}

\begin{observation}\label{para:assumptions}
    For readers familiar with sheaf theory, we point out that all general results pertaining to cohomology that we develop in this paper apply to any site $(\cat{C}, J)$ where $J$ is a Grothendieck pretopology whose covers are \textit{finitely generated}.\footnote{Recall that a sieve on an object $X$ is finitely generated if it can be generated by pre-composition starting with a family of morphisms $(U_i \to X)_{i \in I}$ indexed by a finite family $I$.} For the reader not familiar with sheaf theory, the previous sentence roughly means that we need not restrict to (pre)sheaves of the form $\sub(G)^{\op} \to \Set$; instead we can consider any category $\cat{C}$ with sufficiently nice structure (i.e. admitting pullbacks) and equipped with a well-defined notion of what it means for a collection of objects to ``\textit{cover}'' any given input object $x \in \cat{C}$. We also can consider any Abelian category instead of $\cat{Ab}$, such as $\cat{Vect}(K)$ vector spaces over a field $K$ and $\cat{Mod}(R)$ modules over a ring $R$. In that case, the quotient in Definition~\ref{def:cohomology-cover-specific} can be generalized by $\coker\bigl(\imageObj(\delta^{n-1}) \to \ker(\delta^n)\bigr)$.
\end{observation}

\begin{example}
    Let's see the zeroth presheaf-\v{C}ech cohomology in a concrete example. Consider the same example of Proposition~\ref{prop: vert cover < k is not sheaf}, where $G$ is a two-vertex complete graph. Let's look at the cover $\Ua = \{\emptyset, H_1, H_2\}$ of $H_3$. First, we compute $\ker(\delta^0)$ (which is exactly the matching families)
$$ \ker (\delta^0) = \{\{\emptyset,\emptyset,\emptyset\},~~~~ \{\emptyset, \{v_1\},\emptyset\},~~~~ \{\emptyset,\emptyset, \{v_2\}\}, ~~~~\{\emptyset, \{v_1\}, \{v_2\}\} \}. $$

Now, we need to describe $\im(\delta^{-1})$. We have
$$\delta^{-1} \colon \mathcal{V}_{\leq 1}(H_3) \to \mathcal{V}_{\leq 1}(\emptyset) \times \mathcal{V}_{\leq 1}(H_1) \times \mathcal{V}_{\leq 1}(H_2) $$
$$ \delta^{-1}\colon \{\emptyset,\{v_1\},\{v_2\}\} \to \{\emptyset\} \times \{\emptyset,\{v_1\}\} \times \{\emptyset,\{v_2\}\} $$
which is given by sending each solution of $\Va_{\leq 1}(H_3)$ to its corresponding restriction to each $H_i$. That way, we have
$$ \im(\delta^{-1}) = \{\{\emptyset,\emptyset,\emptyset\},~~~~ \{\emptyset, \{v_1\},\emptyset\},~~~~ \{\emptyset,\emptyset, \{v_2\}\}\}. $$

Therefore, the quotient $H^0 = \ker(\delta^0) / \im(\delta^{-1})$ is generated by $\{\emptyset, \{v_1\}, \{v_2\}\}$, which is exactly the matching family that fails to have an amalgamation.
\end{example}

\body{We have that $\delta^{-1}$ describes how global sections restrict to local sections. If $\mathcal{V}_{\leq 1}$ were a sheaf, every matching family would be "lifted" to a unique global section, i.e., every matching family (all of $\ker(\delta^0)$) would be in the image of $\delta^{-1}$.}

\body{The generators of $H^0$ indicate the local sections that, despite agreeing on the intersections, fail to be lifted to global sections. Recall that the functor $F$ being lavish means that for each matching family of local solutions there exists at least a solution on the global graph whose restriction to each subgraph corresponds to the local solutions of the matching family. Thus, $H^0$ generally informs us about the existence of amalgamations: when $H^0(X,\Ua,F)$ is trivial for all objects $X$ and all covers of $X$, $F$ is lavish; when $H^0(X,\Ua,F)$ is non-trivial, $F$ fails to be lavish. In short, we have learned that the zeroth presheaf \v{C}ech cohomology can be thought of as a measure of the failure of being lavish.}

\body{
One can distill (see for example Gallier and Quaintance~\cite{GallierQuaintanceBook}) from Definition~\ref{def:cohomology-cover-specific} a cover-independent cohomology object $H^n(X,F)$ defined as 
\begin{equation}\label{eqn:def:cover-indep-cohomology}
    H^n(X, F) \defeq \colim_{\Ua \in \Cov (X)} H^n(X, \Ua, F).    
\end{equation} We will be interested in the functor that this defines when the first argument varies and the second is fixed, namely the presheaf:
\[H^n(-, F) \colon \cat{C}\op \to \cat{Ab}.\]}

\body{We note that we can use the technique of Abelianization to define the functor $\Va_{\leq k}$ in $\cat{Ab}$. As a matter of fact, it can be used in any $F$ defined in $\Set$. We describe this process below.}

 \begin{definition}
 Given a set $S$, we say that a function $a \colon S \to \Z$ with $\im(a) \setminus \{0\}$ finite is a \textbf{linear combination of (elements of) $S$}. We can also represent the linear combinations of $S$ as
 $ a = \sum_{s \in S}~ a_s \cdot s$
 where $a_s := a(s)$ and we denote by $s$ the function $s \colon S \to \Z$ that sends $s$ to $1$ and the rest of the elements of $S$ to $0$. Now, we define the group 
 $\Z[S] = \{a: a ~\text{is a linear combination of}~S\}$ 
 with group operation given by
 $(\sum_{s \in S}~ a_s \cdot s) + (\sum_{s \in S}~ b_s \cdot s) = \sum_{s \in S}~ (a_s + b_s) \cdot s.$
 \end{definition}

 \body{In the case of \textsc{VertexCover} we have
 \[ \Z[\Va_{\leq k}(H)] = \{\sum_{A \in \Va_{\leq k}(H)}~ a_A \cdot A : a \text{ is a linear combination of } \Va_{\leq k}(H)\} \]
 where $A$ denotes the function $A \colon \Va_{\leq k}(H) \to \Z$ that sends $A$ to $1$ and the rest of the elements of $\Va_{\leq k}(H)$ to $0$. In the next section, we use this technique to explicitly describe $H^0$ for \textsc{VertexCover}.
 }
 
\section{The Zeroth \v{C}ech Cohomology of Computational Problems}\label{sec: abstract tools and more examples}

\body{
In this section we will consider other computational problems such as \textsc{CycleCover} and \textsc{OddCycleTransversal}. We will show that the zeroth cohomology group finds other kinds of obstructions. For example, in the case of \textsc{CycleCover} (Proposition~\ref{h0-cycle-cover}), the zeroth cohomology groups detects not just issues of size, as it did in the case of \textsc{VertexCover}, but also the phenomenon of \textit{emergent cycles}: cycles that are only seen when one considers more than one piece of a cover at once.
}

\body{
To do this systematically, we will first develop some general results concerning the cohomology of separated presheaves in the following subsection. These results are of independent interest, but they are also useful since they will allow us to avoid tedious calculations later on and indeed they make it possible to apply cohomological tools to study computational problems without needing to focus too much on the commutative algebra involved. With that in mind, the reader who is not interested in (or not familiar with) category theory can safely skip the proofs in the following subsection and simply consider how the results are used later on. 
}

\subsection{Developing Some Technical Machinery}

\body{We have mentioned that one can turn a presheaf into a sheaf through a construction called \textit{sheafification}. Formally, being an \textit{adjoint functor}, there is a sense in which sheafification sends any presheaf to its ``nearest'' sheaf~\cite{rosiak-book, riehl2017category}. We will see that, in cases where $F$ is a separated presheaf on $\sub(G)$, the sheafification of $F$, denoted by $F^+$, will help us characterize the zeroth presheaf \v{C}ech cohomology: $H^0(X,F)$ will be something analogous to the quotient of $F^+(X)$ by $F(X)$. This result gives us an easier and more conceptual approach to computing $H^0$.}

\begin{observation}
    Once again, nodding to readers knowledgeable in sheaf theory, we point out that the following results all generalize to more interesting domains: rather than studying presheaves on $\sub(G)$, the results of this section apply equally well to any category equipped with a Grothendieck pretopology which is \textit{finitely presented} (meaning that each covering sieve is generated by finitely many morphisms).
\end{observation}

\body{To turn $F$ into a sheaf, the idea is to somehow replace $F(H)$ by the matching families over all possible covers of $H$. Sheafification consists of two application of a construction called the \textit{plus construction}. Roughly the idea is the following: we want to keep any section of $F(H)$ that was already an amalgamation of a matching family, but we also want to formally add to our collection of global sections all those matching families that did not admit an amalgamation. Furthermore, we wish to do so without creating duplicates of amalgamations. To that end one instead works with equivalence classes of matching families (where these equivalence classes are obtained by identifying matching families that have a restriction in common).
} \label{idea-sheafification}

\begin{definition}\label{def: sheafification}
    Let $\Sh(\sub(G))$, $\SPsh(\sub(G))$ and $\Psh(\sub(G))$ denote the categories of sheaves, separated presheaves and presheaves on $\sub(G)$, respectively. The composite
    \[
        \Sh(\sub(G)) \hookrightarrow \SPsh(\sub(G)) \hookrightarrow \Psh(\sub(G))
    \]
    admits a left adjoint called the \textbf{sheafification} functor (see Rosiak's textbook~\cite[Sec. 10.2.7]{rosiak-book} for a reference). This functor is given by two applications of the plus construction $(-)^+ \colon \Psh(\sub(G)) \to \Psh(\sub(G))$ defined as 
    \[
        F^+ \colon X \mapsto \underset{\mathcal{U \in \cat{Cov}(X)}}{\cat{colim}} \Match(\Ua,F).
    \]
\end{definition}

\body{In $\Set$, this colimit is given by the quotient 
\[
    \coprod_{\Ua \in \Cov(X)} \Match(\Ua,F)/\sim
\]
where $\coprod$ denotes a disjoint union and $\sim$ is the equivalence relation given as follows: for $A = \{A_i\}_{i \in I} \in \Match(\Ua,F)$ and $B = \{B_j\}_{j \in J} \in \Match(\Ua',F)$, with $\Ua = \{H_i\}_{i \in I}, \Ua' = \{H_j\}_{j \in J} \in \Cov(X)$, we have
\[
    A \sim B ~~\text{iff}~~ \exists~ \Ua''=\{H_l\}_{l \in L}  \in \Cov, ~\text{with}~ \Ua''\subseteq \Ua, \Ua'' \subseteq \Ua', (H),~\text{such that}~\forall l \in L, ~A_l = B_l
\]
which justifies the idea presented in \ref{idea-sheafification}.
} \label{colim-in-set}

\begin{observation}\label{observation:colimit-of-kernels} 
One can also state the definition of the plus construction as
$F^+ X \cong \colim_{\Ua \in \Cov (X)} \ker \delta^0_{\Ua}$. This follows directly from the definition since
$
    \Match(\Ua, F) \cong \ker \delta^0_{\Ua}
$
for any cover $\Ua$ of any $X$.
\end{observation}

\body{
We point out that, for any presheaf $F$, the presheaf $F^+$ is necessarily separated and $F^{++}$ is a sheaf; in particular,
\textit{one application of the plus construction suffices to obtain a sheaf from any separated presheaf} (see \cite{rosiak-book} for a reference). }

\body{Many interesting problems are monotone under the subgraph relation, as we will see later in this section. We will see that such problems can be represented by presheaves $F$ on $\sub(G)$ such that the set of solutions of a subgraph $H \subseteq G$ is actually a subset of $\sub(H)$ (in the case of \textsc{VertexCover}, the vertex covers can be seen as subgraphs with only vertices, no edges). We will also see that, as is the case of \textsc{VertexCover}, such problems are always separated presheaves, where the union is the only possible solution. We prove now a result that can characterize the sheafification of such problems. 

\begin{proposition} \label{prop:chac-sheafification}
    Let $F\colon \sub(G)^{\op} \to \Set$ be a separated presheaf. Suppose that $\tilde{F}\colon \sub(G)^{\op} \to \Set$ is a presheaf such that for all edges $\alpha \in E(G)$, $\tilde{F}(\alpha) \subseteq F(\alpha)$. Moreover, given $H \subseteq G$, suppose that $\bigcup_{i \in I} A_i \in \tilde{F}(H)$ whenever $\{A_i \in F(H_i)\}_{i \in I}$ is a matching family of the cover $H = \bigcup_{i \in I} H_i$. Then $\tilde{F}$ is the sheafification of $F$.
\end{proposition}

\begin{proof}
    Since $F$ is separated we have
    \[ F^+(H) = \colim_{\Ua \in \text{Cov}(H)} \Match(\Ua,F) \]
    So we need to check that 
    \[\tilde{F}(H) = \colim_{\Ua \in \text{Cov}(H)} \Match(\Ua, F)\]
    Given $\Ua,\mathcal{U'} \in \text{Cov}(H)$ and an arrow $\Ua \hookrightarrow \Ua'$, we have the following commutative triangles
    \[
\begin{tikzcd}
{\Match(\Ua,F)} \arrow[rr, "\restrict_{\mathcal{U'}}"] \arrow[rd] &         & {\Match(\mathcal{U'},F)} \arrow[ld] \\
                                                                                                 & \tilde{F}(H) &                                                     
\end{tikzcd}\]
    where the arrows to $\tilde{F}(H)$ act as taking the union. We need to show that, given any $X$ and any arrows $f_{\Ua}: \Match(\Ua,F) \to X$ and $f_{\mathcal{U'}}:\Match(\mathcal{U'},F) \to X$, there exists a unique $k: \tilde{F}(H) \to X$ such that the following commutes
    \[
\begin{tikzcd}
{\Match(\Ua,F)} \arrow[rr, "\restrict_{\mathcal{U'}}"] \arrow[rd] \arrow[rdd, "f_{\Ua}"', bend right] &                                         & {\Match(\mathcal{U'},F)} \arrow[ld] \arrow[ldd, "f_{\mathcal{U'}}", bend left] \\
                                                                                                                               & \tilde{F}(H) \arrow[d, "\exists! k", dashed] &                                                                                   \\
                                                                                                                               & X                                       &                                                                                  
\end{tikzcd}\]
    We divide $H$ in subgraphs given by the edges, i.e., if $E(H) = \{\alpha_1,...,\alpha_n\}$ we denote $H_i = H[\alpha_i]$ and $f_i:H_i \hookrightarrow H$. We then consider the cover $\Ua = \{H_i\}_{i \leq n}$.

    Given $A \in \tilde{F}(H)$, for each $i \leq n$ we consider $A_i = \tilde{F}(f_i)(A)$. By hypothesis, $A_i \in F(H_i)$, for all $i \leq n$. We denote this matching family by
    \[ Y_{H} = \{A_i \in F(H_i)\} \in \Match(\Ua,F) \]

    We then define
    \[\begin{array}{ccccc}
    & k: & \tilde{F}(H) & \longrightarrow         &  X\\
    \\
    &         & A              & \longmapsto &  f_{\Ua}(Y_{H})
    \end{array}\]
    
    Now, because $\Ua$ is the minimal cover of $H$ and because the arrows between the sets of matching families are the restriction, we have that $k$ makes the above diagram commute. Let us see that is the only arrow with this property. 

    Let $k': \tilde{F}(H) \to X$ be a function such that the following diagram commutes ($\Ua$ being the cover we defined above)
    \[
\begin{tikzcd}
{\Match(\Ua,F)} \arrow[rd, "f_{\Ua}"'] \arrow[r, "\cup"] & \tilde{F}(H) \arrow[d, "k'"] \\
                                                                                           & X                      
\end{tikzcd}\]
    Given $H' \subseteq H$ we can always consider the matching family $Y_{H'}$ defined above and we have that
    \[ k'(H') = k'(\cup Y_{H'}) = f_{\Ua}(Y_{H'}) \]
    where the second equality comes from the last commutative diagram. Therefore $k = k'$. 
\end{proof}

\body{Applying the above result, we will now see that, for $k\geq 2$, the sheafification of the presheaf $\Va_{\leq k}$ encoding \textsc{VertexCover} is none other than the sheaf $\Va$ of Proposition~\ref{prop: V is a sheaf}}

\begin{proposition} \label{sheafification-of-Vk}
    $\mathcal{V} \colon \sub(G)^{\cat{op}} \to \Set$ is the sheafification of $\mathcal{V}_{\leq k}$, for $k \geq 2$.
\end{proposition}

\begin{proof}
    We have seen in Proposition~\ref{prop: V is a sheaf} that if we have $H \subseteq G$, a cover $\bigcup_{i \in I}H_i = H$ and $\{A_i\}_{i \in I}$ a family of vertex covers that agree on the intersections, then $\bigcup_{i \in I} A_i$ is a vertex cover of $H$. Furthermore, given an edge $\alpha\in E(G)$ and $A \in \Va(\alpha)$, we have $|A| \leq 2 \leq k$, so that $A \in \Va_{\leq k}$. By Proposition~\ref{prop:chac-sheafification}, $\Va$ is the sheafification of $\Va_{\leq k}$, whenever $k \geq 2$.
\end{proof}

\body{One of the contributions of this work is to show that this connection between cohomology and sheafification can be made very explicit and satisfying: Theorem~\ref{thm:cokernel} will show that, for any presheaf $F$, the zeroth presheaf Čech cohomology functor arises as the quotient of $F^+$ by $F$ (precisely, it will be the cokernel of $\eta_F \colon F \to F^+$, the unit of the adjunction given by sheafification).}

\begin{lemma}\label{lemma:colimit-of-images}
If $F \colon \sub(G)^{\op} \to \cat{Ab}$ is separated, then $\colim_{\Ua \in \Cov(X)} \im \delta^{-1}_{\Ua} \cong F(X)$.
\end{lemma}
\begin{proof}
Since $F$ is separated, the morphism $\delta^{-1}_{\Ua} \colon F(X) \to \prod_{i} F(H_i)$ is mono as shown in the following commutative diagram. 
\[\begin{tikzcd}
	{F(X)} && {\prod_{i}F(H_i)} \\
	& {\ker \delta^0_{\Ua}}
	\arrow["{\delta^{-1}_{\Ua}}", hook, from=1-1, to=1-3]
	\arrow["{\xi_{\Ua}}"', hook, from=1-1, to=2-2]
	\arrow[hook, from=2-2, to=1-3]
\end{tikzcd}\]
Thus, since $F$ is valued in an Abelian category, where every morphism which is both a monomorphism and an epimorphism is already an isomorphism, one has that $F(X) \cong \imageObj  \delta^{-1}_{\Ua}$. Moreover, this made no assumption on the choice of $\Ua$, thus 
\begin{equation}\label{eqn:lemma:colimit-of-images}
F(X) \cong \imageObj  \delta^{-1}_{\Ua} \quad \forall \Ua \in \Cov(X).     
\end{equation}
By the property of colimits, there exists a unique arrow $r \colon \colim_{\Ua \in \Cov(X)} \im \delta^{-1}_{\Ua} \to F(X)$ such that for all $\Ua \in \Cov(X)$
\[
\begin{tikzcd}
\imageObj \delta^{-1}_{\Ua} \arrow[r, "i"] \arrow[d, "s"']      & F(X) \\
\colim_{\Ua \in \Cov(X)} \im \delta^{-1}_{\Ua} \arrow[ru, "r"'] &     
\end{tikzcd}\]
commutes, that is, $r \circ s = i$, where $i \colon \imageObj \delta^{-1}_{\Ua} \cong F(X)$. Let us see that $r$ is an isomorphism. Since $i$ is iso, there exists $j \colon F(X) \to \imageObj \delta^{-1}_{\Ua}$ such that $i \circ j = id_{F(X)}$ and $j \circ i =id_{\imageObj \delta^{-1}_{\Ua}}$.
We have that $s \circ j$ is the inverse of $r$. Indeed, we have that
\[ r \circ s \circ j = i \circ j = id_{F(X)}. \]
Now, by the property of colimits, $id_{\imageObj \delta^{-1}_{\Ua}}$ is the only arrow such that
\[
\begin{tikzcd}
\imageObj \delta^{-1}_{\Ua} \arrow[r, "s"] \arrow[d, "s"']       & \colim_{\Ua \in \Cov(X)} \im \delta^{-1}_{\Ua} \\
\colim_{\Ua \in \Cov(X)} \im \delta^{-1}_{\Ua} \arrow[ru, "id"'] &                                               
\end{tikzcd}\]
commutes, that is, $id_{\imageObj \delta^{-1}_{\Ua}} \circ s = s$. But, 
\[ s \circ j \circ r \circ s = s \circ j \circ i = s \circ id_{\imageObj \delta^{-1}_{\Ua}} = s \]
so that $s \circ j \circ r = id_{\imageObj \delta^{-1}_{\Ua}}$. Therefore, $\colim_{\Ua \in \Cov(X)} \im \delta^{-1}_{\Ua} \cong F(X)$, as desired. 
\end{proof}

\body{
We will now use Observation~\ref{observation:colimit-of-kernels} and Lemma~\ref{lemma:colimit-of-images} to prove that, if $F$ is separated, then the functor $H^0(-,F) \colon \sub(G)^{\op} \to \cat{Ab}$ can be rewritten as the quotient of $F^+$ by $F$.
}

\begin{corollary}\label{corollary:cokernel}
Let $F \colon \sub(G)^{\op} \to \Set$ be a separated presheaf. For every $X \subseteq G$, we fix a cover $\Ua \in \Cov(X)$ and consider $f_X \colon F(X) \to F^+(X)$ given by 
\[
    F(X) \xrightarrow{\xi_{\Ua}} \Match(\Ua,F) \hookrightarrow \coprod_{\Ua \in \Cov(X)} \Match(\Ua,F) \xrightarrow{\pi} F^+(X)
\]
where $\pi \colon \displaystyle\coprod_{\Ua \in \Cov(X)} \Match(\Ua,F) \to F^+(X) = \coprod_{\Ua \in \Cov(X)} \Match(\Ua,F)/\sim$ is the projection map of the equivalence relation defined in \ref{colim-in-set}.
Then, we have that the zeroth \v{C}ech cohomology is the quotient
$$H^0(X,F) = \Z(F^+(X)) / \Z(\im f_X)$$
\end{corollary}

\body{
    The proof of Corollary~\ref{corollary:cokernel} is rather algebraic and thus mostly relevant to readers comfortable with sheaf theory. For this reason and since the proof techniques actually yield a much more general result, we will only give a proof of this stronger theorem (Theorem~\ref{thm:cokernel} below) from which one can easily deduce Corollary~\ref{corollary:cokernel} as a special case.
}

\begin{theorem}\label{thm:cokernel}
Suppose $F$ is an Abelian presheaf on a site $(\cat{C}, J)$ where $J$ is a finitely generated Grothendieck topology. If $F$ is separated, then, letting $\eta \colon F \Rightarrow F^+$ be the unit of the adjunction given by sheafification, one has that $H^0(-,F) = \coker (F \xrightarrow{\eta_F} F^+)$. 
\end{theorem}
\begin{proof}
We have
    \begin{align*}
        H^0(X,F) &\defeq \colim_{\Ua \in \Cov(X)} H^0(X, \Ua, F) &\text{(Equation~\eqref{eqn:def:cover-indep-cohomology})} \\
        &= \colim_{\Ua \in \Cov(X)} \coker \bigl( \im \delta^{-1} \xrightarrow{\xi^0_{\Ua}} \ker \delta^0_{\Ua} \bigr) &\text{(definition of } H^0) \\
        &= \coker \colim_{\Ua \in \Cov(X)} \bigl( \im \delta^{-1} \xrightarrow{\xi^0_{\Ua}} \ker \delta^0_{\Ua} \bigr) &\text{(colimits commute)} \\
        &= \coker \bigl( \colim_{\Ua \in \Cov(X)} \im \delta^{-1} \xrightarrow{\colim_{\Ua \in \Cov(X)} \xi^0_{\Ua}} \colim_{\Ua \in \Cov(X)} \ker \delta^0_{\Ua} \bigr) \\
        &= \coker \bigl( F X \xrightarrow{ \colim_{\Ua \in \Cov(X)} \xi^0_{\Ua}} \colim_{\Ua \in \Cov(X)} \ker \delta^0_{\Ua} \bigr) &\text{(Lemma~\ref{lemma:colimit-of-images})} \\
        &= \coker \bigl( F X \xrightarrow{ \colim_{\Ua \in \Cov(X)} \xi^0_{\Ua}} F^+ X \bigr) &\text{(Observation~\ref{observation:colimit-of-kernels})} \\
    \end{align*}
The morphisms $\colim_{\Ua \in \Cov(X)} \xi^0_{\Ua}$ are the components of a natural transformation $\eta \colon F \Rightarrow F^+$. Moreover, this is precisely the definition of the unit of the sheafification functor. Thus, our previous derivation allows us to restate the functor $H^0(-,F)$ as a cokernel of $\eta$, as desired. 
\end{proof}

\body{
    In the example of Proposition~\ref{prop: vert cover < k is not sheaf}, $\Va_{\leq k}$ failed to be a sheaf when the union of local solutions was ``too big''. We will see now that this case is the only obstruction we have for this particular presheaf. First, we prove a result that will help even more with our calculation of the zeroth cohomology. Since all the problems we work with in this paper are separated presheaves, they are sub-presheaves of their sheafification (i.e. the unit of the adjunction is monic). In other words, one has that the set of solutions $F(H)$ is canonically a subset of $F^+(H)$. With this in mind, the next result shows that, when dealing with presheaves of sets, which are Abelianized freely, the quotient of $F^+(H)/F(H)$ is equivalent to the Abelianization of the difference $F^+(H) - F(H)$.
}

\begin{lemma} \label{coker-set-complement}
    Let $Y$ be a set and $B \subseteq Y$ a subset. If $f: B \hookrightarrow Y$ is the inclusion, then the quotient $\Z[Y] / \Z[\im f]$ is isomorphic to $\Z[Y - B]$.
\end{lemma} 

\begin{proof}
    To make the writing simpler, we denote $W:= \Z[\im f]$. We define 
    \[ \begin{array}{ccccc}
    & h: & \Z[Y] / W & \longrightarrow         &  \Z[Y - B]\\
    \\
    &         &     \displaystyle\sum_{y \in Y} a_y \cdot y + W        & \longmapsto &  \displaystyle\sum_{y \in Y-B} a_y \cdot y
    \end{array} \]
    First, let's see that $h$ is well defined. Let $\sum_{y \in Y} a_y \cdot y + W = \sum_{y \in Y} b_y \cdot y + W$. We have that $0 \in W$, so there exists $\sum_{y \in B} c_y \cdot y \in W$ such that
    \[ \sum_{y \in Y} a_y \cdot y + 0 = \sum_{y \in Y} b_y \cdot y + \sum_{y \in B} c_y \cdot y \]
    We can rewrite that as
    \[ \displaystyle\sum_{y \in B} a_y \cdot y + \displaystyle\sum_{y \in Y-B} a_y \cdot y = \displaystyle\sum_{y \in B} (b_y + c_y) \cdot y + \displaystyle\sum_{y \in Y-B} b_y \cdot y\]
    so that $\sum_{y \in Y-B} a_y \cdot y = \sum_{y \in Y-B} b_y \cdot y$, which means that $\sum_{y \in Y} a_y \cdot y + W$ and $\sum_{y \in Y} b_y \cdot y + W$ are sent by $h$ to the same element.

    Let's check that $h$ is injective. Suppose that $\sum_{y \in Y-B} a_y \cdot y = \sum_{y \in Y-B} b_y \cdot y$, let's see that $\sum_{y \in Y} a_y \cdot y + W = \sum_{y \in Y} b_y \cdot y + W$. Let $\sum_{y \in Y} a_y \cdot y + \sum_{y \in B} c_y \cdot y$ be an arbitrary element of $\sum_{y \in Y} a_y \cdot y + W$, we consider $\sum_{y \in B} (a_y +c_y - b_y) \cdot y \in W$, so that
    \[ \displaystyle\sum_{y \in Y} b_y \cdot y + \displaystyle\sum_{y \in B} (a_y +c_y - b_y) \cdot y = \displaystyle\sum_{y \in B} (a_y + c_y) \cdot y + \displaystyle\sum_{y \in Y-B} b_y \cdot y \]
    and using the hypothesis we have
    \[\displaystyle\sum_{y \in B} (a_y + c_y) \cdot y + \displaystyle\sum_{y \in Y - B} b_y \cdot y = \displaystyle\sum_{y \in B} (a_y + c_y) \cdot y + \displaystyle\sum_{y \in Y-B} a_y \cdot y = \displaystyle\sum_{y \in Y} a_y \cdot y + \displaystyle\sum_{y \in B} c_y \cdot y \]
    and thus we proved that $\sum_{y \in Y} a_y \cdot y + W = \sum_{y \in Y} b_y \cdot y + W$.

    Finally, let's check that $h$ is surjective. Given $\sum_{y \in Y-B} a_y \cdot y$, we define $\sum_{y \in Y} b_y \cdot y \in \Z[Y]$ as
    \[b_y= \left\{ \begin{array}{ll}
    a_y , & \text{if } y \in Y-B\\
    0 , & \text{if } y \in Y
\end{array} \right.\]
    and we have that $h(\sum_{y \in Y} b_y \cdot y) = \sum_{y \in Y-B} a_y \cdot y$.
\end{proof}

\begin{observation}
    Notice that the fact that $F^+(H)/F(H)$ is equivalent to the Abelianization of $F^+(H) - F(H)$ is possible precisely because in this paper we are considering free Abelianizations of \textit{presheaves of sets}. We leave it as exciting future work to consider Abelian presheaves where the algebraic structure is not free.
\end{observation}

\body{We now have all tools to finally characterize the zeroth \v{C}ech cohomology for \textsc{VertexCover}.}

\begin{proposition} \label{h0-vertex-cover}
    If $k\geq 2$, then $H^0(X,\Va_{\leq k}) = \Z[\{A \in \Va(X) : |A| > k\}]$. 
\end{proposition}

\begin{proof}
    We fix a cover $\Ua = \{X_i\}_{i \leq m} \in \Cov(X)$ and we consider $f_X \colon \Va_{\leq k}(X) \to \Va(X)$ given by 
    \[ \Va_{\leq k}(X) \xrightarrow{\xi_{\Ua}} \Match(\Ua,\Va_{\leq k}) \hookrightarrow \coprod_{\Ua \in \Cov(X)} \Match(\Ua,\Va_{\leq k}) \xrightarrow{\pi} \Va_{\leq k}^+(X) \cong \Va(X)\]
     On an object $X \in \sub(G)$, this function identifies sections that agree locally on some cover of $X$. The action of $f_X$ is given by
     \[ A \in \Va_{\leq k}(X) \xmapsto{\xi_{\Ua}} (A\restrict_{X_{1}},...,A\restrict_{X_{m}}) \hookrightarrow (A\restrict_{X_{1}},...,A\restrict_{X_{m}}) \xmapsto{\pi} [(A\restrict_{X_{1}},...,A\restrict_{X_{m}})] \xmapsto{\cong} \bigcup_{i \leq m} A\restrict_{X_{i}} = A\]
    Thus, $f_X$ can be seen as the inclusion $f_X \colon \Va_{\leq k}(X) \hookrightarrow \Va(X)$.
    
    Defining $Y = \Va(X)$ and $B = \Va_{\leq k}(X)$, we have that 
    $$\Z[Y-B] = \Z[\{A \in \Va(X): |A|>k\}].$$ We then use Theorem~\ref{thm:cokernel} and Lemma~\ref{coker-set-complement} to obtain the desired result.
\end{proof}

\body{
Summarizing briefly, as we mentioned earlier, $H^{0}$ intuitively collects all the obstructions to lavishness. Based on the example of Proposition~\ref{prop: vert cover < k is not sheaf}, we argued informally that these obstructions should only have to do with cardinality: namely, in the case of \textsc{VertexCover} the only obstructions to algorithmic compositionality are the sizes of the associated sets. We showed this formally by proving that $H^0(-, \mathcal{V}_{\leq k})$ is exactly $\Z[\{A \in \Va: |A|>k\}]$. However, these are not the only kinds of obstructions that $H^0$ can detect: in the following subsection we will see that, along with detecting ``size issues'', $H^0$ also detects other kinds of \textit{emergent phenomena}. For example, in the case of \textsc{CycleCover} we can divide a graph in pieces in such way that $H^0$ will detect cycles that are not visible locally by pieces that are small. 
}

\subsection{More computational problems as examples}

\body{In this section, we spell out another two concrete examples and consider their cohomology.  Furthermore, since we use constructions similar to those employed for \textsc{VertexCover}, we will employ some general category-theoretic tools that, as biproduct, can be easily seen to yield a great many more examples of computational problems as presheaves on $\sub(G)$. Since these examples would all have a similar flavor, we limit ourselves to giving concrete descriptions of \textsc{VertexCover}, \textsc{CycleCover} and \textsc{BipartiteCover}.}

\subsection{Categorical Generalities to Construct Presheaves from Monotone Graph Properties}

\begin{definition}
    Given a category $\cat{C}$ and a fixed object $X$ of $\cat{C}$, we say that an object $Y$ of $\cat{C}$ is a \textbf{subobject} of $X$ if there exists in $\cat{C}$ a monic arrow $f \colon Y \hookrightarrow X$.
\end{definition}

\begin{definition}
    A functor $p \colon \mathbb{P} \hookrightarrow \cat{C}$ is \textbf{pullback-absorbing} if, given any pullback square of the following form,
    \[
\begin{tikzcd}
p(X) \times_Y Z \arrow[r, hook] \arrow[d, hook] & Z \arrow[d, hook] \\
p(X) \arrow[r, hook]                            & Y                
\end{tikzcd}\]
    there exists an object $X' \in \mathbb{P}$ such that $p(X') \cong p(X) \times_Y Z$.
\end{definition}

\body{We intend to use this language in the following way: we wish to use the functor $p \colon \mathbb{P} \hookrightarrow \cat{C}$ to talk about a property that some objects of $\cat{C}$ satisfy. Given a property $P$, we see $\mathbb{P}$ as the category of objects of $\cat{C}$ that satisfy $P$, and $p(X)$ will denote that $X$ satisfies the property $P$.}

\body{Intuitively a property $p$ is pullback-absorbing if for all $Y \in \cat{C}$, the pullback of two subobjects of $Y$ satisfies the property as long as one of the subobjects of $Y$ satisfies the property. In the case $\cat{C} = \sub(G)$, this translates to \textit{monotonicity under subobjects}: if an object $Y$ satisfies the property $p$, all the subobjects of $Y$ also satisfy $p$.}

\begin{lemma}
    If $\cat{C} = \sub(G)$, then $p$ is pullback-absorbing if and only if $\mathbb{P}$ is monotone under the subgraph relation. 
\end{lemma}

\begin{proof}
    $(\Rightarrow)$ Let $Y=p(Y) \in \mathbb{P}$ and $f \colon Z \hookrightarrow Y$ be a subobject of $y$. Consider the following pullback 
    \[
    \begin{tikzcd}
    p(Y) \times_Y Z \arrow[r, hook] \arrow[d, hook] & Z \arrow[d, "f", hook] \\
    p(Y) \arrow[r, "id"', hook]                     & p(Y)=Y                
    \end{tikzcd}\]
    we have that $id\colon Z \to Z$ and $f \colon Z \hookrightarrow Z$ make the following commute
    \[
\begin{tikzcd}
Z \arrow[r, "id", hook] \arrow[d, "f"', hook] & Z \arrow[d, "f", hook] \\
p(Y) \arrow[r, "id"', hook]                   & p(Y)=Y                
\end{tikzcd}\]
    so that there exists a unique $k \colon Z \to p(Y) \times_Y Z$ such that the whole diagram commutes
    \[
\begin{tikzcd}
Z \arrow[rrd, "id", bend left] \arrow[rdd, "f"', bend right] \arrow[rd, "k", dashed] &                                                 &                        \\
                                                                                     & p(Y) \times_Y Z \arrow[r, hook] \arrow[d, hook] & Z \arrow[d, "f", hook] \\
                                                                                     & p(Y) \arrow[r, "id"', hook]                     & p(Y)=Y                
\end{tikzcd}\]
    Therefore, $p(Y) \times_Y Z = Z$, and because $p$ is pullback-absorbing, we have that $p(Y) \times_Y Z = p(Z)$.

    $(\Leftarrow)$ When we consider the pullback
    \[
\begin{tikzcd}
p(X) \times_Y Z \arrow[r, hook] \arrow[d, hook] & Z \arrow[d, hook] \\
p(X) \arrow[r, hook]                            & Y                
\end{tikzcd}\]
we have that $p(X) \times_Y Z$ is a subobject of $p(X)$, so, by hypothesis, it also satisfies $p$.
\end{proof}

\begin{lemma} \label{functor} If $p_1 \colon \mathbb{P}_1 \hookrightarrow \sub(G)$ and $p_2 \colon \mathbb{P}_2 \hookrightarrow \sub(G)$ are pullback-absorbing functors, then the following is a presheaf
\[\begin{array}{ccccc}
    & \sub_{\mathbb{P}_{1,2}}(-) \colon & \sub(G)^{op} & \longrightarrow         &  \Set\\
    \\
    &         & H              & \longmapsto &  \{H' \subseteq H: H' \in \mathbb{P}_1 \text{ and } H-H' \in \mathbb{P}_2\}
    \end{array}\]
    and given $f \colon H \hookrightarrow K$, we put
\[\begin{array}{ccccc}
    & \sub_{\mathbb{P}_{1,2}}(f) \colon & \sub_{\mathbb{P}_{1,2}}(K) & \longrightarrow         &  \sub_{\mathbb{P}_{1,2}}(H)\\
    \\
    &         & K'              & \longmapsto &  K'\cap H
    \end{array}\]
    where $K' \cap H$ is the pullback of $K' \hookrightarrow K$ and $H \hookrightarrow K$ in the category $\sub(G)$. 
\end{lemma}

\begin{proof}
    We have that $\sub_{\mathbb{P}_{1,2}}(-)$ is well defined on the arrows because $(K - K') \cap H = H - (K' \cap H)$ and because $p_1$ and $p_2$ are pullback-absorbing, so that $K' \cap H \in \mathbb{P}_1$ and $H - (K' \cap H) \in \mathbb{P}_2$. Also, it is a functor, because if $H \hookrightarrow J \hookrightarrow K$ and
    \[
\begin{tikzcd}
L =K' \times_K J \arrow[r, hook] \arrow[d, hook] & J \arrow[d, hook] &  & L \times_J H \arrow[r, hook] \arrow[d, hook] & H \arrow[d, hook] \\
K' \arrow[r, hook]                            & K                 &  & L \arrow[r, hook]                            & J                
\end{tikzcd}\]
    are pullbacks, then the bigger square
    \[
\begin{tikzcd}
L \times_J H \arrow[r, hook] \arrow[d, hook] & H \arrow[d, hook] \\
L =K' \times_K J \arrow[r, hook] \arrow[d, hook] & J \arrow[d, hook] \\
K' \arrow[r, hook]                            & K                
\end{tikzcd}\]
    is a pullback, which is the same as saying that $K' \cap H = K' \times_K H = (K' \times_K J) \times_J H = (K' \cap J)\cap H$ and means that the following commutes
    \[
\begin{tikzcd}
\sub_{\mathbb{P}_{1,2}}(K) \arrow[r] \arrow[rr, bend right] & \sub_{\mathbb{P}_{1,2}}(J) \arrow[r] & \sub_{\mathbb{P}_{1,2}}(H)
\end{tikzcd}\]
    And since $H' \cap H = H'$, for $H' \subseteq H$, we also have that $\sub_{\mathbb{P}_{1,2}}(id_H) = id_{\sub_{\mathbb{P}_{1,2}}(H)}$.
\end{proof}

\begin{lemma} \label{separated}
    $\sub_{\mathbb{P}_{1,2}}(-)$ is a separated presheaf.
\end{lemma}

\begin{proof}
    Let $H$ be a subgraph of $G$, $\{H_i\}_{i \in I}$ a cover of $H$ and $H_1,H_2 \in \sub_{\mathbb{P}_{1,2}}(H)$, ie, $H_1,H_2 \subseteq H$ such that $H_1,H_2 \in \mathbb{P}_{1}$, and suppose that $H_1\restrict_{H_{i}} = H_2\restrict_{H_{i}}$ for all $i \in I$. Let us see that $H_1=H_2$. If $v \in V(H_1) \subseteq V(H)$, then there exists $i$ such that $v \in V(H_i)$, so that $v \in V(H_1)\restrict_{H_{i}} = V(H_2)\restrict_{H_{i}}$, therefore $v \in V(H_1)$. We can prove that all edges of $H_1$ are in $H_2$ using the same argument and the inclusion $H_2 \subseteq H_1$ is analogous.
\end{proof}

\begin{remark} \label{rmk: one-property-pullback}
    For the purpose of this work, having in mind the examples we will be working on, we defined the previous functor with a property for the complement of graphs, but we note that if we define $H \mapsto \{H' \subseteq H: H' \in \mathbb{P}\}$, with only one property about the subgraphs, this also results in a separated presheaf, as long as $p \colon \mathbb{P} \hookrightarrow \sub(G)$ is pullback-absorbing, since we can put $\mathbb{P}_2 = \sub(G)$.
\end{remark}

\body{For \textsc{VertexCover}, we define $\mathbb{V} \subseteq \sub(G)$  as $H \in \mathbb{V}$ iff $H$ is edgeless; and given a $k \in \mathbb{N}$, we define $\mathbb{V}_k \subseteq \sub(G)$ as $H \in \mathbb{V}_k$ iff $H$ is edgeless and $|H| \leq k$. We have that $\mathbb{V}$ and $\mathbb{V}_{k}$ are monotone under subobjects.} \label{property mathbbV}

\body{Notice that $\Va_{\leq k}$ can be alternatively defined as
\[\begin{array}{ccccc}
    & \mathcal{V}_{\leq k} \colon & \sub(G)^{\op} & \longrightarrow         &  \Set\\
    \\
    &         & H              & \longmapsto &  \{H' \subseteq H: H' \in \mathbb{V}_k\text{ and } H - H' \in \mathbb{V}\}
    \end{array}\]
    and the relationship between the set of solutions of subgraphs is described as follows: given $f \colon H \hookrightarrow K$, we put
\[\begin{array}{ccccc}
    & \Va_{\leq k}(f) \colon & \Va_{\leq k}(K) & \longrightarrow         &  \Va_{\leq k}(H)\\
    \\
    &         & K'              & \longmapsto &  K'\cap H
    \end{array}\]}

\begin{lemma} \label{vc-k}
    $\Va_{\leq k}$ is a separated functor.
\end{lemma}

\begin{proof}
    Since $\mathbb{V}$ and $\mathbb{V}_k$ are monotone under subobjects, the result follows from Lemma~\ref{functor} and Lemma~\ref{separated}.
\end{proof}

\subsection{More Examples: \textsc{CycleCover} and \textsc{OddCycleTransversal}}

\begin{definition}
    Given a graph $G = (V(G),E(G))$, we say that $S \subseteq V(G)$ is a \textbf{cycle cover of} $G$ if $G - S$ is a forest, i.e., it does not have cycles. We say that $S$ is a \textbf{minimum cycle cover} if, for every cycle cover $R$, $|S| \leq |R|$.
\end{definition}

\body{We can state the cycle cover problem as}

\vspace{0.4cm}

\fbox{\begin{minipage}{41em}
$\textsc{CycleCover}$

\textbf{Input:} A graph $G$ and an integer $k$.

\textbf{Task:} Decide if $G$ admits a cycle cover of cardinality at most $k$.
\end{minipage}}

\body{We will use the same property of \ref{property mathbbV}, $\mathbb{V}_k$, that relates to being a set of vertices of a given cardinality. But for talking about cycle covers, we need another property: we define $\mathbb{C} \subseteq \sub(G)$ as $H \in \mathbb{C}$ iff $H$ doesn't have a cycle. We have that $\mathbb{C}$ is monotone under subobjects.}

\body{We define
\[\begin{array}{ccccc}
    & \mathscr{C}_{\leq k} \colon & \sub(G)^{\op} & \longrightarrow         &  \Set\\
    \\
    &         & H              & \longmapsto &  \{H' \subseteq H: H' \in \mathbb{V}_{k} ~\text{and }H-H' \in \mathbb{C}\}
    \end{array}\]
and given $f\colon H \hookrightarrow K$, we put
\[\begin{array}{ccccc}
    & \mathscr{C}_{\leq k}(f)\colon & \mathscr{C}_{\leq k}(K) & \longrightarrow         &  \mathscr{C}_{\leq k}(H)\\
    \\
    &         & K'              & \longmapsto &  K' \cap H
    \end{array}\]
}

\begin{lemma}
    $\mathscr{C}_{\leq k}$ is a separated functor.
\end{lemma}

\begin{proof}
    Since $\mathbb{V}_k$ and $\mathbb{C}$ are monotone under subobjects, the result follows from Lemma~\ref{functor} and Lemma~\ref{separated}.
\end{proof}

\begin{lemma}
    $\mathscr{C}_{\leq k}$ is not a sheaf.
\end{lemma}

\begin{proof}

Just like in the case of vertex covers, the only possible gluing for a matching family is the union. What might happen is that the union doesn't need to be a cycle cover as we see in this example: let's consider $H = K_3$ and the following cover $\{H_1, H_2\}$:
\[
\tikzset{every picture/.style={line width=0.75pt}} 
\begin{tikzpicture}[x=0.75pt,y=0.75pt,yscale=-1,xscale=1]

\draw    (212,66) -- (248.6,147.6) ;
\draw    (212,66) -- (166.6,140.6) ;
\draw [shift={(166.6,140.6)}, rotate = 121.32] [color={rgb, 255:red, 0; green, 0; blue, 0 }  ][fill={rgb, 255:red, 0; green, 0; blue, 0 }  ][line width=0.75]      (0, 0) circle [x radius= 3.35, y radius= 3.35]   ;
\draw [shift={(212,66)}, rotate = 121.32] [color={rgb, 255:red, 0; green, 0; blue, 0 }  ][fill={rgb, 255:red, 0; green, 0; blue, 0 }  ][line width=0.75]      (0, 0) circle [x radius= 3.35, y radius= 3.35]   ;
\draw    (166.6,140.6) -- (248.6,147.6) ;
\draw [shift={(248.6,147.6)}, rotate = 4.88] [color={rgb, 255:red, 0; green, 0; blue, 0 }  ][fill={rgb, 255:red, 0; green, 0; blue, 0 }  ][line width=0.75]      (0, 0) circle [x radius= 3.35, y radius= 3.35]   ;
\draw    (423,56) -- (459.6,137.6) ;
\draw [shift={(459.6,137.6)}, rotate = 65.84] [color={rgb, 255:red, 0; green, 0; blue, 0 }  ][fill={rgb, 255:red, 0; green, 0; blue, 0 }  ][line width=0.75]      (0, 0) circle [x radius= 3.35, y radius= 3.35]   ;
\draw    (423,56) -- (377.6,130.6) ;
\draw [shift={(377.6,130.6)}, rotate = 121.32] [color={rgb, 255:red, 0; green, 0; blue, 0 }  ][fill={rgb, 255:red, 0; green, 0; blue, 0 }  ][line width=0.75]      (0, 0) circle [x radius= 3.35, y radius= 3.35]   ;
\draw [shift={(423,56)}, rotate = 121.32] [color={rgb, 255:red, 0; green, 0; blue, 0 }  ][fill={rgb, 255:red, 0; green, 0; blue, 0 }  ][line width=0.75]      (0, 0) circle [x radius= 3.35, y radius= 3.35]   ;
\draw    (374.6,174.6) -- (456.6,181.6) ;
\draw [shift={(456.6,181.6)}, rotate = 4.88] [color={rgb, 255:red, 0; green, 0; blue, 0 }  ][fill={rgb, 255:red, 0; green, 0; blue, 0 }  ][line width=0.75]      (0, 0) circle [x radius= 3.35, y radius= 3.35]   ;
\draw [shift={(374.6,174.6)}, rotate = 4.88] [color={rgb, 255:red, 0; green, 0; blue, 0 }  ][fill={rgb, 255:red, 0; green, 0; blue, 0 }  ][line width=0.75]      (0, 0) circle [x radius= 3.35, y radius= 3.35]   ;

\draw (135,138) node [anchor=north west][inner sep=0.75pt]   [align=left] {$\displaystyle v_{1}$};
\draw (218,38) node [anchor=north west][inner sep=0.75pt]   [align=left] {$\displaystyle v_{3}$};
\draw (257,145) node [anchor=north west][inner sep=0.75pt]   [align=left] {$\displaystyle v_{2}$};
\draw (470,77) node [anchor=north west][inner sep=0.75pt]   [align=left] {$\displaystyle H_{1}$};
\draw (478,180) node [anchor=north west][inner sep=0.75pt]   [align=left] {$\displaystyle H_{2}$};
\end{tikzpicture}
\]
Now, we have that the empty set is a solution for the subgraphs $H_1$ and $H_2$, but the union is empty, which is not a solution for $H$.
\end{proof}

\body{Let's think about the zeroth presheaf-\v{C}ech cohomology in the example given in the last proof. Calculating the matching families (i.e. $\ker \delta^0$) and the restrictions of $\mathscr{C}(H)$ (i.e. $\im\delta^1$), we have that $\{\emptyset, ..., \emptyset\}$ is the only matching family that doesn't lift to a global solution. This corresponds to the fact that $H^0$ is detecting the presence of an \textit{emergent cycle}: the two graphs $H_1$ and $H_2$ do not contain cycles; and yet, their union does.}

\body{Since $\mathscr{C}_{\leq k}$ is separated, we can characterize its zeroth presheaf-\v{C}ech cohomology as we did for \textsc{VertexCover}. First, we need to talk about the sheafification of $\mathscr{C}_{\leq k}$.}

\begin{proposition}
    For $k \geq 2$, the sheafification of $\mathscr{C}_{\leq k}$ is given by
\[\begin{array}{ccccc}
    & \sub(-) \colon & \sub(G)^{\op} & \longrightarrow         &  \Set\\
    \\
    &         & H              & \longmapsto &  \{H': H' \subseteq H\}
    \end{array}\]
and given $f\colon H \hookrightarrow K$, we put
\[\begin{array}{ccccc}
    & \sub(f)\colon & \sub(K) & \longrightarrow         &  \sub(H)\\
    \\
    &         & K'              & \longmapsto &  K'\cap H
    \end{array}\]
\end{proposition}

\begin{proof}
    Given $H \subseteq G$, a cover $\bigcup_{i \in I} H_i = H$ and $\{H'_i \in \mathscr{C}_{\leq k}(H_i)\}_{i \in I}$ a matching family of cycle covers, we have $\bigcup_{i \in I} H'_i \in \sub(H)$. Moreover, given $\alpha = K \in E(H)$ an edge, if we consider $f: K \hookrightarrow H$, then 
    $$\sub(f)(H') = H' \cap K, ~~\forall H' \in \sub(H).$$
    Since $K$ is an edge, we have that $H' \cap K$ is a cycle cover of $K$ and $|H' \cap K| \leq 2 \leq k$, so that $\sub(f)(H') \in \mathscr{C}_{\leq k}(H)$. By Proposition~\ref{prop:chac-sheafification}, we have that $\sub(-)$ is the sheafification of $\mathscr{C}_{\leq k}.$
\end{proof}

\begin{proposition} \label{h0-cycle-cover}
    $H^0(X,\mathscr{C}_{\leq k}) = \Z[\{X' \subseteq X: X' \notin \mathbb{V}_k ~\text{or}~ X-X' \notin \mathbb{C}\}]$. 
\end{proposition}
\begin{proof}
    We fix a cover $\Ua = \{X_i\}_{i \leq m} \in \Cov(X)$ and we consider $f_X \colon \Va_{\leq k}(X) \to \Va(X)$ given by 
    \[ \mathscr{C}_{\leq k}(X) \xrightarrow{\xi_{\Ua}} \Match(\Ua,\mathscr{C}_{\leq k}) \hookrightarrow \coprod_{\Ua \in \Cov(X)} \Match(\Ua,\mathscr{C}_{\leq k}) \xrightarrow{\pi} \mathscr{C}_{\leq k}^+(X) \cong \sub(X)\]
    As was the case for \textsc{VertexCover} in Proposition~\ref{h0-vertex-cover}, we have that $f_X$ can be seen as the inclusion $f_X \colon \mathscr{C}_{\leq k}(X) \hookrightarrow \sub(X)$.
    
    Defining $Y = \sub(X)$ and $B = \mathscr{C}_{\leq k}(X)$, we have that 
    $$\Z[Y-B] = \Z[\{X' \subseteq X: X' \notin \mathbb{V}_k ~\text{or}~ X-X' \notin \mathbb{C}\}].$$ We then use Theorem~\ref{thm:cokernel} and Lemma~\ref{coker-set-complement} to obtain the desired result. 
\end{proof}

\body{With this result we see that, in the case of \textsc{CycleCover}, $H^0$ can detect not only the problem with size but also emergent cycles.}

\begin{definition}
    Given a graph $G$, we say that $S \subseteq V(G)$ is a \textbf{bipartite cover of} $G$ if $G - S$ is a bipartite graph. We say that $S$ is a \textbf{minimum bipartite cover} if, for every bipartite cover $R$, $|S| \leq |R|$.
\end{definition}

\body{We can state the bipartite cover problem (also known as \textsc{OddCycleTransversal}) as follows.}

\vspace{0.4cm}

\fbox{\begin{minipage}{41em}
$\textsc{BipartiteCover}$

\textbf{Input:} A graph $G$ and an integer $k$.

\textbf{Task:} Decide if $G$ admits a bipartite cover of cardinality at most $k$.
\end{minipage}}

\body{Again, we will use the property $\mathbb{V}_{\leq k}$ presented in \ref{property mathbbV}. We also define other property $\mathbb{B}$ as follows: $H \in \mathbb{B}$ iff $H$ is bipartite. We have that $\mathbb{B}$ is monotone under subobjects.}

\body{We define
\[\begin{array}{ccccc}
    & \mathscr{B}_{\leq k}\colon & \sub(G)^{\op} & \longrightarrow         &  \Set\\
    \\
    &         & H              & \longmapsto &  \{H' \subseteq H: H' \in \mathbb{V}_{k} ~\text{and }H-H' \in \mathbb{B}\}
    \end{array}\]
and given $f\colon H \hookrightarrow K$, we put
\[\begin{array}{ccccc}
    & \mathscr{B}_{\leq k}(f)\colon & \mathscr{B}_{\leq k}(K) & \longrightarrow         &  \mathscr{B}_{\leq k}(H)\\
    \\
    &         & K'              & \longmapsto &  K' \cap H
    \end{array}\]
}

\body{And the next results follow as in the case of \textsc{CycleCover}, including its sheafification and the example that $\mathscr{C}_{\leq k}$ is not a sheaf.}

\begin{lemma}
    $\mathscr{B}_{\leq k}$ is a separated functor but not a sheaf.
\end{lemma}

\begin{lemma}
    $H^0(X,\mathscr{B}_{\leq k}) = \Z[\{X' \subseteq X: X' \notin \mathbb{V}_k ~\text{or}~ X-X' \notin \mathbb{B}\}]$.
\end{lemma}

\subsection{The Compositionality of $H^0$}

\body{
We saw that the zeroth presheaf \v{C}ech cohomology provides the information of the obstructions to naive algorithmic compositionality: i.e., the obstructions to the success of naive dynamic programming algorithms. Given a computational problem $F$, it is natural to ask if these failures, namely $H^0(-, F)$ itself, admit some form of well-behaved compositionality. In other words, we ask: "how close to being a sheaf can $H^0$ itself be?"
}

\body{
We will see that although $H^0(-, F)$ is not a sheaf in general, if $F$ has a suitably nice structure, then $H^0(-, F)$ is a \textit{lavish} presheaf which, recalling from earlier (c.f. Definition~\ref{sep-lav-set}), is to say that it satisfies the existence condition for global section. 
}

\body{
Fixing some Abelian category $\cat{A}$ (e.g. $\cat{Ab}$ or $\cat{Vect}_{\mathbb{R}}$), we will now employ Theorem~\ref{thm:cokernel} to show that $H^0(-,F)$ is lavish for an Abelian presheaf $F \colon \sub(G)^{\op} \to \cat{A}$ whenever $F$ is \define{flasque}, meaning that its restriction maps are \textit{epimorphisms}. To this end, we will first need the following technical lemma which shows that one can distribute cokernels over pullbacks in an Abelian category. We prove this result for $\cat{A}$ \textit{any} Abelian category; the reader not familiar with the relevant categorical notions can simply think of the whole proof as concerning Abelian groups.
}
\begin{lemma}\label{lemma:cokernels-quasi-distribute-over-pullback}
In an Abelian category, for any morphism $f$ into a pullback as shown in the following diagram,  
\[\begin{tikzcd}[column sep=small]
	&&& L \\
	Z && {L\times_MR} && M \\
	&&& R
	\arrow["{\ell_1}", from=2-3, to=1-4]
	\arrow["{\ell_2}", from=1-4, to=2-5]
	\arrow["{r_1}"', from=2-3, to=3-4]
	\arrow["{r_2}"', from=3-4, to=2-5]
	\arrow["f"', from=2-1, to=2-3]
\end{tikzcd}\]
there is an epimorphism $u:\coker f \twoheadrightarrow \coker(\ell_1 f) \times_{\coker(\ell_2\ell_1 f)} \coker(r_1 f)$.
\end{lemma}
\begin{proof}
We liberally assume to be working in a concrete Abelian category of modules, by Freyd--Mitchell embedding \cite{mitchell-embedding}.

From the situation described, one always gets a commutative square of cokernels which, by universality of pullbacks, yields a unique arrow as dashed below:
\[
\adjustbox{scale=1.5, max width=.95\textwidth}{
\begin{tikzcd}[ampersand replacement=\&]
	\&\&\& {\coker(\ell_1f)} \&[-6ex] \\
	{\coker f} \&\& {\coker(\ell_1f) \times_{\coker(\ell_2\ell_1f)}\coker(r_1f)} \&\& {\coker(\ell_2\ell_1f)} \\
	\&\&\& {\coker(r_1f)}
	\arrow["{\ell_2^*}", from=1-4, to=2-5]
	\arrow["{r_2^*}"', from=3-4, to=2-5]
	\arrow["{\ell_1^*}", curve={height=-12pt}, from=2-1, to=1-4]
	\arrow["{r_1^*}"', curve={height=12pt}, from=2-1, to=3-4]
	\arrow[from=2-3, to=1-4]
	\arrow[from=2-3, to=3-4]
	\arrow["\lrcorner"{anchor=center, pos=0.125, rotate=45}, draw=none, from=2-3, to=2-5]
	\arrow["{\exists!\, u}"{description}, dashed, from=2-1, to=2-3]
\end{tikzcd}
}
\]
Here $(-)^*$ denotes the maps induced between cokernels.

Any element of $\coker(\ell_1 f) \times_{\coker(\ell_2\ell_1 f)} \coker(r_1 f)$ is a coset of the form
\[
    (S_L, S_R) := (\lambda + \imageObj(\ell_1f), \rho + \imageObj(r_1f))
\]
for some $\lambda \in L$ and $\rho \in R$ such that $\ell_2^*(S_L) = r_2^*(S_R)$. 
Thus the coset $(\lambda, \rho) + \im f \in \coker f$ maps to $(S_L, S_R)$, which proves $u$ is an epimorphism.
\end{proof}

\begin{theorem}\label{thm:lavish}
Suppose $F$ is an Abelian presheaf on a site $(\cat{C},J)$ where $J$ is a finitely generated Grothendieck topology. If $F$ is flasque and separated, then $H^0(-,F) \colon \cat{C}\op \to \cat{A}$ is lavish with respect to $J$.
\end{theorem}
\begin{proof}
Since $F$ is separated, a single application of the plus construction amounts to sheafification. With this in mind, for any cover $(U_i \to X)_{i\in I}$ of $X$ in $J$ one has 
\begin{align*}
    &H^0(X, F) = \coker (F X \xrightarrow{\eta_{F_X}} F^+ X) &(\text{Theorem~\ref{thm:cokernel}})\\
    &\cong \coker \bigl( F X \xrightarrow{\eta_{F_X}} \ker( \prod_i F^+ U_i \xrightarrow{p-q} \prod_{i,j} F^+ U_{i,j}) \bigr) &(F^+ \text{ is a sheaf})\\
    &\twoheadrightarrow \ker \bigl( \coker(F X \to \prod_i F^+ U_i) \to \coker(F X \to \prod_{i,j} F^+ U_{i,j}))\bigr) &\text{(Lemma~\ref{lemma:cokernels-quasi-distribute-over-pullback})} \\ 
    &= \ker \bigl( \coker(\prod_i (F X \to  F^+ U_i)) \to \coker(\prod_{i,j} (F X \to  F^+ U_{i,j})))\bigr) \\
\intertext{which, since colimits commute with each other, since \textit{finite} products and \textit{finite} coproducts agree in Abelian categories and since the covers in $J$ are \textit{finitely generated} (see Paragraph~\ref{para:assumptions}), yields}
    &= \ker \bigl( \prod_i \coker(F X \to  F^+U_i) \to \prod_{i,j} \coker(F X \to F^+ U_{i,j}))\bigr) \\
    &= \ker \bigl( \prod_i \coker(F X \twoheadrightarrow F U_i \to  F^+ U_i) \to \prod_{i,j} \coker(F X  \twoheadrightarrow F U_{i,j} \to F^+ U_{i,j}))\bigr) &(\eta\text{ is natural})\\
    &\cong \ker \bigl( \prod_i \coker(F U_i \to  F^+ U_i) \to \prod_{i,j} \coker(F U_{i,j} \to F^+ U_{i,j}))\bigr) &\text{($F$ is flasque)}\\
    &= \ker \bigl( \prod_i H^0(U_i, F) \to \prod_{i,j} H^0(U_{i,j}, F)\bigr) &(\text{Theorem~\ref{thm:cokernel}}).
\end{align*}
\end{proof}

\body{
Although none of the problems that we saw so far are flasque, we will show in the next section that there is a functorial construction that sends any presheaf to a \textit{flasque and separated} presheaf that records relevant combinatorial information and this construction can be directly applied to the old \textsc{VertexCover}, \textsc{CycleCover} and \textsc{BipartiteCover} problems -- as well as any problems that can be represented by presheaves on $\sub(G)$.
}

\section{Model collecting}

\body{In this section, instead of considering the obstructions to ``algorithmic compositionality'' (i.e. obstructions to being a sheaf), we will focus on identifying suitable abstract language to speak about obstructions to having \textit{any} solution at all. Given a presheaf $F$ that describes a computational problem, we will define another presheaf $\mathfrak{M}F$ which will send any object $H$ to the set 
\[\{H' \subseteq H \mid F(H') \neq \emptyset\}\]
of all subgraphs of $H$ such that which are yes-instances for $F$-\textsc{Decision}. In a way, $\mathfrak{M}F(H)$ can be thought of as \textit{collecting} all the subgraphs of $H$ that are \textit{models}\footnote{Our motivation for this terminology is the idea that when we have a formula $p$ describing a property in the logic of graphs, we can consider the set $\{G' \hookrightarrow G : G' \vDash p\}$ which will be the models for $p$.} for $F$. The rest of this section is devoted to the proof of the following fact.}

\begin{theorem}\label{thm: model collecting}
    There exists a covariant functor $\mathfrak{M}$, called the \textbf{model collecting functor}, that maps any presheaf $F \colon \cat{C}^{\op} \to \cat{Set}$, where $\cat{C}$ is a category with pullbacks, to a flasque presheaf $\mathfrak{M}F$. Moreover, if $\cat{C} = \sub(G)$ and the complete graph with two vertices has a solution for $F$, then
    \begin{itemize}
        \item $H^0(X,\mathfrak{M}F) = \Z[\{X' \subseteq X: F(X') = \emptyset\}]$;
        \item $FX \neq \emptyset$ iff $H^0(X,\mathfrak{M}F) = 0$.
    \end{itemize}
\end{theorem}

\body{The above result tells us we can use cohomology to decide if a graph $X$ has a solution for a certain problem $F$. For example consider König's Theorem which states that a bipartite graph $X$ admits a matching of size greater than $k$ if and only if $X$ has no vertex cover with size at most $k$. This theorem can be restated in terms of the zeroth presheaf cohomology of the model collecting functor, as follows.}

\begin{theorem}[König's Theorem, Cohomologically]
    If $X$ is a bipartite graph, then $$H^0(X,\mathfrak{M}\mathcal{V}_{\leq k}) = \{\text{matchings of }X \text{ of size greater than }k\}.$$
\end{theorem}

\body{Although we don't provide an algebraic proof of König's Theorem, the above result points to the fact that cohomological methods provide a possible avenue towards stating and proving results in extremal graph theory. However, this is beyond the scope of the present paper and we leave it as future work.}

\subsection{Proof of Theorem~\ref{thm: model collecting}}
\body{Given a presheaf $F \colon \cat{C}^{\cat{op}} \to \Set$, we define
\[\begin{array}{ccccc}
    & \mathfrak{M}F \colon & \cat{C}^{\cat{op}} & \longrightarrow         &  \Set\\
    \\
    &         & C              & \longmapsto &  \{C' \hookrightarrow C: F(C') \neq \emptyset \}
    \end{array}\]
and given $f \colon C \hookrightarrow D$, we put

\[\begin{array}{ccccc}
    & \mathfrak{M}F(f) \colon & \mathfrak{M}F(D) & \longrightarrow         &  \mathfrak{M}F(C)\\
    \\
    &         & D'              & \longmapsto &  D' \times_D C
    \end{array}\]
where $D' \times_D C$ is the pullback of $D' \hookrightarrow D$ and $C \hookrightarrow D$ in the category $\cat{C}$.
}

\begin{lemma} \label{lemma:mc-is-functorial}
    $\mathfrak{M} \colon \Psh(\cat{C}) \to \Psh(\cat{C}))$ is a covariant functor.
\end{lemma}

\begin{proof}
    Given $\alpha \colon F \Rightarrow L$ a natural transformation between two functors $F,L\colon \cat{C}^{\op} \to \Set$, we can define a natural transformation $\mathfrak{M}(\alpha)\colon \mathfrak{M}F \to \mathfrak{M}L$ that acts as an identity, since, given $C' \in \mathfrak{M}F(C)$, we have an arrow $\alpha_{C'}: F(C') \to L(C')$ with $F(C') \neq \emptyset$, so that $C' \in \mathfrak{M}L(C)$ and $\mathfrak{M}(\alpha)$ is then well defined.
    With this definition, the property of naturality of $\mathfrak{M}(\alpha)$ and the property of $\mathfrak{M}$ being a functor follow easily.
\end{proof}

\begin{lemma} \label{lemma:mc-funtor}
    For every presheaf $F\colon \cat{C}^{\op} \to \Set$, $\mathfrak{M}F$ is also a presheaf.
\end{lemma}

\begin{proof}
   We define the above construction with a property $\mathbb{P}$ defined as follows: $C$ is an object of $\mathbb{P}$ iff $F(C) \neq \emptyset$. We then consider the functor 
\[\begin{array}{ccccc}
    & \cat{C}_{\mathbb{P}} \colon & \cat{C}^{\op} & \longrightarrow         &  \Set\\
    \\
    &         & C              & \longmapsto &  \{C' \hookrightarrow C: C' \in \mathbb{P} \}
    \end{array}\]
    and given $f \colon C \hookrightarrow D$, we put
\[\begin{array}{ccccc}
    & \cat{C}_{\mathbb{P}}(f) \colon & \cat{C}_{\mathbb{P}}(D) & \longrightarrow         &  \cat{C}_{\mathbb{P}}(C)\\
    \\
    &         & D'              & \longmapsto &  D'\times_D C
    \end{array}\]
    
    Now, $p\colon \mathbb{P} \hookrightarrow \cat{C}$ pullback-absorbing because if 
    \[
\begin{tikzcd}
p(X) \times_Y Z \arrow[r, hook] \arrow[d, hook] & Z \arrow[d, hook] \\
p(X) \arrow[r, hook]                            & Y                
\end{tikzcd}\]
    is a pullback in $\cat{C}$, then
\[
\begin{tikzcd}
F(p(X) \times_Y Z) & F(Z) \arrow[l]           \\
F(p(X)) \arrow[u]  & F(Y) \arrow[l] \arrow[u]
\end{tikzcd}\]
    is a diagram in $\Set$. Since $F(X) \neq \emptyset$ and we have an arrow $F(X) \to F(X \times_Y Z)$, we have that $F(X \times_Y Z) \neq \emptyset$, so that $X \times_Y Z $ is an object of $\mathbb{P}$.

    We have that $\cat{C}_{\mathbb{P}}$ is well defined on the arrows because $p$ is pullback-absorbing. Also, it is a functor, because if $C \hookrightarrow D \hookrightarrow E$ and
    \[
\begin{tikzcd}
A = E' \times_E D \arrow[r, hook] \arrow[d, hook] & D \arrow[d, hook] &  & A \times_D C \arrow[r, hook] \arrow[d, hook] & C \arrow[d, hook] \\
E' \arrow[r, hook]                            & E                &  & A \arrow[r, hook]                            & D                
\end{tikzcd}\]
    are pullbacks, then the bigger square
    \[
\begin{tikzcd}
A' \times_D C \arrow[r, hook] \arrow[d, hook] & C \arrow[d, hook] \\
A= E' \times_D E \arrow[r, hook] \arrow[d, hook] & D \arrow[d, hook] \\
E' \arrow[r, hook]                            & E                
\end{tikzcd}\]
    is a pullback, which means that the following commutes
    \[
\begin{tikzcd}
\cat{C}_{\mathbb{P}}(E) \arrow[r] \arrow[rr, bend right] & \cat{C}_{\mathbb{P}}(D) \arrow[r] & \cat{C}_{\mathbb{P}}(C)
\end{tikzcd}\]
    And since $C' \times_C C \cong C'$, for $C' \hookrightarrow C$, we also have that $\cat{C}_{\mathbb{P}}(id_C) = id_{\cat{C}_{\mathbb{P}}(C)}$.
\end{proof}

\begin{lemma}\label{lemma:mc-factors-through-flasque}
    If we denote by $\cat{FPSh}(\cat{C})$ the category of flasque presheaves, then $\mathfrak{M}$ factors through $\cat{FPSh}(\cat{C})$, i.e., the following diagram commutes. 
    \[
\begin{tikzcd}
\Psh(\cat{C}) \arrow[rr, "\mathfrak{M}"] &                                          & \Psh(\cat{C}) \\
                                       & \cat{FPSh}(\cat{C}) \arrow[lu] \arrow[ru, hook] &            
\end{tikzcd}\]
\end{lemma}

\begin{proof}
We need to check that, for every $f \colon C \rightarrow D$, $\mathfrak{M}F(f)$ is surjective. Given $C' \in \mathfrak{M}F(C)$, ie, $C' \xrightarrow{g} C$ such that $F(C') \neq \emptyset$, we have
\[
\begin{tikzcd}
C' \arrow[r, "g"] & C \arrow[d, "f"] \\
                        & D                     
\end{tikzcd}\]
so that we have an arrow $C' \rightarrow D$ with $F(C') \neq \emptyset$. Hence, $C' \in \mathfrak{M}F(D)$ and $\mathfrak{M}F(f)(C') = C'$
\end{proof}

\begin{lemma}\label{lemma:mc-factors-through-separated}
    If we denote by $\SPsh(\cat{C})$ the category of separated presheaves, then, when $\cat{C} = \sub(G)$, we have that $\mathfrak{M}$ factors through $\SPsh(\cat{C})$, i.e., the following diagram commutes. 
    \[
\begin{tikzcd}
\Psh(\sub(G)) \arrow[rr, "\mathfrak{M}"] &                                          & \Psh(\sub(G)) \\
                                       & \SPsh(\sub(G)) \arrow[lu] \arrow[ru, hook] &            
\end{tikzcd}\]
\end{lemma}

\begin{proof}
    We note that to prove Lemma \ref{separated} we didn't use the properties $\mathbb{P}_1$ and $\mathbb{P}_2$, only the topology defined in $\sub(G)$ and the fact that elements of $\sub_{\mathbb{P}_{1,2}}(H)$ are subgraphs of $H$, so we can use the same argument to prove that $\cat{C}_{\mathbb{P}}$ is a separated presheaf when $\cat{C} = \sub(G)$.
\end{proof}

\body{Now, we wish to look at the zeroth presheaf \v{C}ech cohomology for $\mathfrak{M}F$ in the case where $\cat{C} = \sub(G)$. Again, since $\mathfrak{M}F$ is a separated presheaf, we can use Theorem~\ref{thm:cokernel} to characterize $H^0(X,\mathfrak{M}F)$ as the quotient $\mathfrak{M}F^+ /\mathfrak{M}F$. So the next step will be to describe the sheafification of $\mathfrak{M}F$.}

\begin{proposition} \label{model-collecting-sheafification}
    Suppose $F: \sub(G)^{\op} \to \Set$ is a presheaf such that $F(K_2) \neq \emptyset$, i.e., the complete graph with two vertices has a solution for $F$, then $(\mathfrak{M}F)^+ = \sub(-)$.
\end{proposition}

\begin{proof}
    Given $H \subseteq G$, a cover $\bigcup_{i \in I}H_i = H$ and $\{H'_i \in \mathfrak{M}F(H_i)\}_{i \in I}$ a matching family, we have $\bigcup_{i \in I} H'_i \in \sub(H)$. Furthermore, given an edge $\alpha = K \in E(H)$, if we consider $f \colon K \hookrightarrow H$, we have
    \[ \sub(f)(H') = H' \cap K, ~~\forall H' \in \sub(H). \]
    Now, because $K \cong K_2$, by hypothesis we have $F(K) \neq \emptyset$. Since $H' \cap K \hookrightarrow K$ and the function $F(K) \to F(H' \cap K)$ is in $\Set$, we have $F(H' \cap K) \neq \emptyset$, so that $\sub(f)(H') \in \mathfrak{M}F(H)$. By Proposition~\ref{prop:chac-sheafification}, we have that $\sub(-)$ is the sheafification of $\mathfrak{M}F$, whenever $F(K_2) \neq \emptyset$.
\end{proof}

\body{Note that the hypothesis $F(K_2)\neq 0$ is verified by all the problems we have been working on so far. Thus, the next two characterizations of $H^{0}(-,\mathfrak{M}F)$ can be applied when $F$ describes the problems of \textsc{VertexCover}, \textsc{CycleCover} and \textsc{BipartiteCover}.}

\begin{proposition} \label{h0-model-collecting}
    If $F(K_2) \neq \emptyset$, then $H^{0}(X,\mathfrak{M}F) = \Z[\{X' \subseteq X: F(X') = \emptyset\}]$.
\end{proposition}

\begin{proof}
   We fix a cover $\Ua = \{X_i\}_{i \leq m} \in \Cov(X)$ and we consider $f_X \colon \Va_{\leq k}(X) \to \Va(X)$ given by 
    \[ \mathfrak{M}F(X) \xrightarrow{\xi_{\Ua}} \Match(\Ua,\mathfrak{M}F) \hookrightarrow \coprod_{\Ua \in \Cov(X)} \Match(\Ua,\mathfrak{M}F) \xrightarrow{\pi} \mathfrak{M}F^+(X) \cong \sub(X)\]
    As it was the case for \textsc{VertexCover} and \textsc{CycleCover} in Proposition~\ref{h0-vertex-cover} and Proposition~\ref{h0-cycle-cover}, we have that $f_X$ can be seen as the inclusion $f_X \colon \mathfrak{M}F(X) \hookrightarrow \sub(X)$.
    
    Defining $Y = \sub(X)$ and $B = \mathfrak{M}F(X)$, we have that 
    $$\Z[Y-B] = \Z[\{X' \subseteq X: F(X') = \emptyset\}].$$ We then use Theorem~\ref{thm:cokernel} and Lemma~\ref{coker-set-complement} to obtain the desired result.
\end{proof}

\body{We now have a characterization the helps us proving the main result of this section.}

\begin{theorem}
        $FX \neq \emptyset$ iff $H^0(X,\mathfrak{M}F) = 0$, whenever $F(K_2) \neq \emptyset$.
\end{theorem}

\begin{proof}
    Using Proposition~\ref{h0-model-collecting} we only need to prove that
    \[ FX \neq \emptyset \Longleftrightarrow \forall X' \subseteq X, FX' \neq \emptyset \]
    The $(\Leftarrow)$ direction is trivial, since $X \subseteq X$. To check $(\Rightarrow)$, we note that if $X' \subseteq X$ then we have the restriction function $FX \to FX'$ in $\Set$ with $FX \neq \emptyset$, which implies that $FX' \neq \emptyset$. 
\end{proof}

\section{Discussion \& Future Work}

\body{This work takes a first step toward applying cohomological tools to the study of obstructions in algorithmics and combinatorics. Building on prior work~\cite{deciding-sheaves} that interprets dynamic programming as a sheaf condition -- that is, the ability to assemble global solutions from compatible local ones -- we model computational problems as functors, or presheaves, assigning sets of certificates to input instances. Within this categorical framework, we show that two key types of algorithmic obstructions -- failures of solution existence and failures of compositionality -- can be naturally detected and expressed via Čech cohomology. In particular, the zeroth cohomology group captures the emergence of global inconsistencies not visible at the local level.}

\body{
    By adapting Čech cohomology to preshaves, our work illustrates how existing cohomological methods can be meaningfully applied to classical algorithmic problems. We explore this perspective in detail through three concrete cases -- \textsc{VertexCover}, \textsc{CycleCover}, and \textsc{OddCycleTransversal} -- and show how their structural obstructions manifest cohomologically. Although our examples are specific, the approach generalizes to a broader class of problems involving monotone graph properties, suggesting wider applicability. These observations invite further exploration at the intersection of topology, category theory, and algorithm design.
}

\body{
Looking ahead, a natural next step is to investigate the role of higher Čech cohomology groups in detecting
more subtle obstructions within computational problems. Higher cohomology is defined in the usual way, as \(H^n(X, \Ua, F) = \coker(\imageObj(\delta^{n-1}) \to \ker(\delta^n))\). To explore the potential of higher cohomology in this setting, one can construct a short exact sequence of presheaves 
\[
0 \to \Va_{\leq k} \hookrightarrow \mathscr{C}_{\leq k} \twoheadrightarrow \mathscr{C}_{\leq k} / \Va_{\leq k} \to 0,
\]
relating the problems studied in this paper. This induces a long exact sequence of cohomology functors as usual. 
However, when cohomology is computed by freely Abelianizing these set-valued presheaves, the connecting map 
\[
H^0(-, \mathscr{C}_{\leq k} / \Va_{\leq k}) \to H^1(-, \Va_{\leq k}).
\]
is trivial -- again suggesting that deeper insights may require defining computational problems directly as presheaves valued in an Abelian category.
}

\bibliography{biblio}
\bibliographystyle{acm}
\end{document}